\documentclass[11pt]{amsart}
\usepackage[utf8]{inputenc}

\usepackage{overpic, amssymb, times, graphicx, tikz-cd}
\usepackage[all]{xy}
\usepackage[backref=page]{hyperref}
\usepackage{verbatim}
\hypersetup{
    colorlinks=true,
    linkcolor=blue,
    filecolor=blue,      
    urlcolor=blue,
    citecolor=blue,
}
\numberwithin{equation}{subsection}

\DeclareMathOperator{\eval}{ev}
\DeclareMathOperator{\ord}{ord}
\DeclareMathOperator{\Aut}{Aut}

\DeclareMathOperator{\Crit}{crit}

\DeclareMathOperator{\im}{im}

\DeclareMathOperator{\Diff}{Diff}
\DeclareMathOperator{\Hom}{Hom}

\DeclareMathOperator{\ind}{ind}

\DeclareMathOperator{\Int}{int}

\DeclareMathOperator{\Maslov}{M}

\DeclareMathOperator{\Flow}{Flow}

\DeclareMathOperator{\tb}{tb}

\DeclareMathOperator{\expdim}{expdim}
\DeclareMathOperator{\SFT}{SFT}

\newcommand{\C}{\mathbb{C}}
\newcommand{\R}{\mathbb{R}}
\newcommand{\Z}{\mathbb{Z}}
\newcommand{\N}{\mathbb{N}}

\newcommand{\B}{\mathbb{B}}

\newcommand{\disk}{\mathbb{D}}

\newcommand{\chord}{r}
\newcommand{\bigO}{\mathcal{O}}

\newcommand{\region}{\mathcal{R}}

\newcommand{\action}{\mathcal{A}}
\newcommand{\energy}{\mathcal{E}}

\newcommand{\delbar}{\overline{\partial}}
\newcommand{\Cinfty}{\mathcal{C}^{\infty}}
\newcommand{\Rthree}{(\R^{3},\xi_{std})}

\newcommand{\Circle}{S^{1}}

\newcommand{\half}{\frac{1}{2}}

\newcommand{\ModSpace}{\mathcal{M}}

\newcommand{\be}{\begin{enumerate}}
\newcommand{\ee}{\end{enumerate}}

\newcommand{\Cfancy}{\mathcal{C}}
\newcommand{\Dfancy}{\mathcal{D}}

\newcommand{\norm}[1]{\left\lVert#1\right\rVert}

\newcommand{\obstruction}{\mathcal{O}}
\newcommand{\fatSigma}{\boldsymbol{\Sigma}}

\newtheorem{thm}{Theorem}[section]

\newtheorem{assump}[thm]{Assumptions}

\newtheorem{prop}[thm]{Proposition}

\newtheorem{defn}[thm]{Definition}

\newtheorem{lemma}[thm]{Lemma}
\newtheorem{cor}[thm]{Corollary}

\newtheorem{q}[thm]{Question}
\newtheorem{rmk}[thm]{Remark}

\newtheorem{edits}[thm]{Editor notes}

\title{Simplified SFT moduli spaces for Legendrian links}
\author{Russell Avdek}
\address{Department of mathematics, Uppsala University, Box 480, 751 06 Uppsala, Sweden}
\thanks{The author is partly supported by the grant KAW 2016.0198 from the Knut and Alice Wallenberg Foundation}
\email{russell.avdek@math.uu.se} 
\date{\today}

\begin{document}
\begin{abstract}
We study moduli spaces $\ModSpace$ of holomorphic maps $U$ to $\R^{4}$ with boundaries on the Lagrangian cylinder over a Legendrian link $\Lambda \subset \Rthree$. We allow our domains, $\dot{\Sigma}$, to have non-trivial topology in which case $\ModSpace$ is the zero locus of an obstruction function $\obstruction$, sending a moduli space of holomorphic maps in $\C$ to $H^{1}(\dot{\Sigma})$. In general, $\obstruction^{-1}(0)$ is not combinatorially computable. However after a Legendrian isotopy $\Lambda$ can be made \emph{left-right-simple}, implying that any $U$
\be
\item of index $1$ is a disk with one or two positive punctures for which $\pi_{\C}\circ U$ is an embedding.
\item of index $2$ is either a disk or an annulus with $\pi_{\C} \circ U$ simply covered and without interior critical points.
\ee
Therefore any SFT invariant of $\Lambda$ is combinatorially computable using only disks with $\leq 2$ positive punctures.
\end{abstract}
\maketitle

\setcounter{tocdepth}{1}
\tableofcontents
\pagebreak

\section{Introduction}

Let $\Rthree$ denote the standard contact structure on $3$-space, $\R^{3} \simeq \R \times \C$. With coordinates $(t, x, y)$,
\begin{equation*}
\xi_{std} = \ker(\alpha_{std}), \quad \alpha_{std} = dt - ydx.
\end{equation*}
A link $\Lambda$ in $\R^{3}$ is \emph{Legendrian} if it is tangent to $\xi_{std}$ in which case we write $\Lambda \subset \Rthree$.

In \cite{Chekanov:LCH, Eliashberg:LCH}, Chekanov and Eliashberg define a combinatorially-computable symplectic field theory (SFT) \cite{EGH:SFTIntro} invariant of isotopy classes of Legendrian links in $\Rthree$, called \emph{Legendrian contact homology}. For a Legendrian link $\Lambda \subset \Rthree$ we will denote this invariant by $LCH(\Lambda)$. See \cite{EtnyreNg:LCHSurvey} for an overview of $LCH$ and related invariants.

The graded algebra $LCH(\Lambda)$ is the homology of a differential graded algebra (DGA) generated by $\partial_{t}$ chords of $\Lambda$. The differential for the DGA counts finite energy holomorphic disks $U$ with boundary on the Lagrangian cylinder
\begin{equation*}
\R \times \Lambda \subset \R_{s} \times \R^{3}.
\end{equation*}
The disks have a single boundary puncture asymptotic to a chord of $\Lambda$ at the $s \rightarrow \infty$ end of $\R \times \Lambda$ and any number of boundary punctures asymptotic to chords at the $s \rightarrow -\infty$ end of $\R \times \Lambda$. In this article, we always use the \emph{standard complex structure}, $J$, defined
\begin{equation}\label{Eq:Jstd}
J \partial_{s} = \partial_{t},\quad J \partial_{x} = \partial_{y}.
\end{equation}
See Section \ref{Sec:SpecifyJ}. We also assume that the domains of holomorphic maps are connected.

In \cite{Avdek:RSFT, Ekholm:Z2RSFT, Ng:RSFT}, the author, Ekholm, and Ng define versions of \emph{Legendrian rational symplectic field theory} ($RSFT$) for Legendrian links by counting holomorphic disks with any numbers of positive and negative punctures. A next logical step in the development of SFT for Legendrian links would be to incorporate counts holomorphic maps $U$ whose domains are any compact, connected Riemann surfaces $(\Sigma, j)$ with boundary punctures removed, $\dot{\Sigma} \subset \Sigma$, allowing multiple boundary components or positive genus. While an invariant which counts all such curves is yet to be rigorously defined for Legendrian links, we'll call such a hypothetical invariant \emph{Legendrian SFT}, and denote it by $LSFT(\Lambda)$.

Chekanov's $LCH$ and Ng's $RSFT$ are defined by counting immersions of disks to $\C$ with boundary on $\pi_{\C}(\Lambda)$, facilitating combinatorial proofs of well-definition and invariance under Legendrian isotopy. The fact that these combinatorial invariants coincide with their analytically-defined counterparts follows from the Riemann mapping theorem coupled with analysis of Etnyre, Ng, and Sullivan \cite[Section 7]{ENS:Orientations}. The goal of this article is to address the possibility of combinatorial counting of holomorphic curves on $\Lambda$ whose domains are not necessarily contractible.

\begin{q}\label{Q:Main}
Can the $LSFT$ moduli spaces $\ModSpace^{\Lambda}$ of holomorphic maps
\begin{equation*}
U = (s, t, u): \dot{\Sigma} \rightarrow \R \times \R \times \C,\quad U(\partial \dot{\Sigma}) \subset \R \times \Lambda
\end{equation*}
be described combinatorially from $\lambda = \pi_{\C}(\Lambda) \subset \C$ when $H^{1}(\dot{\Sigma}) \neq 0$?
\end{q}

\subsection{Main results}

The answer to this question is ``no'' in general but ``yes'' after applying a Legendrian isotopy to $\Lambda$. This first assertion will be justified shortly. The second is the content of the following theorem:

\begin{thm}\label{Thm:Main}
After applying a Legendrian isotopy to the link $\Lambda$, the following conditions apply to holomorphic curves with boundary on $\R \times \Lambda$ and boundary punctures asymptotic to $\partial_{t}$-chords of $\Lambda$:
\be
\item Every $\ind = 1$ holomorphic curve $U$ with boundary on $\R \times \Lambda$ is a disk with $1$ or $2$ positive punctures and such that $\pi_{\C}\circ U$ is an embedding. There are only finitely many such disks up to holomorphic reparameterization and translation in the $s$-coordinate.
\item Every $\ind = 2$ holomorphic curve $U$ with boundary on $\R \times \Lambda$ is either a disk with at most $3$ positive punctures, or an annulus with at most $2$ positive punctures.  The map $\pi_{\C} \circ U$ is simple and without critical points in the interior of its domain.\footnote{We recall -- cf. \cite{MS:Curves} -- that a holomorphic map $u$ is \emph{multiply covered} if it may be written $u = \widetilde{u}\circ \phi$ where $\phi$ is a branched covering of the domain of $u$ onto another Riemann surface, which is the domain of $\widetilde{u}$. \emph{Simple curves} are those which are not multiply covered.} Consequently, $U$ is simple and without interior critical points.
\ee
\end{thm}

\begin{figure}[h]
	\begin{overpic}[scale=.4]{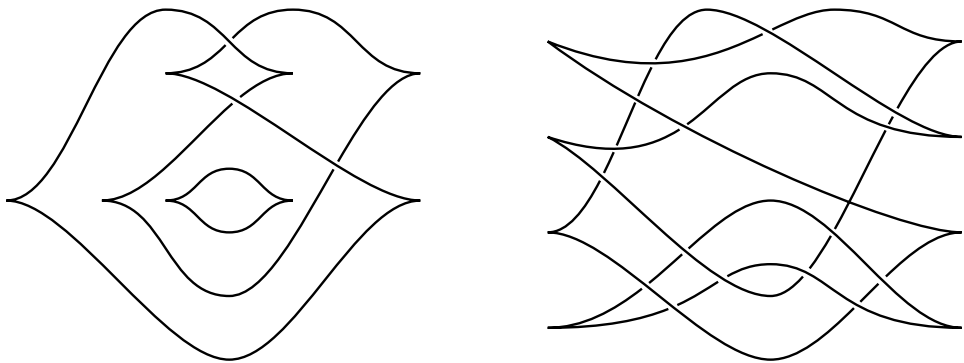}
	\end{overpic}
	\caption{The two-component Legendrian link on the left is modified by Reidemeister moves to obtain the link on the right whose Lagrangian resolution is left-right-simple.}
	\label{Fig:Simplification}
\end{figure}

Theorem \ref{Thm:Main} utilizes a new notion of simplicity for Legendrian links.

\begin{defn}\label{Def:LRS}
Let $\Lambda$ be a Legendrian link in good position (Definition \ref{Def:GoodPosition}) and such that $x|_{\Lambda}$ is Morse.
\be
\item $\Lambda$ is \emph{left-simple} is there exists $x_{L} \in \R$ for which each local minimum of $x|_{\Lambda}$ has $x$-value $x_{L}$.
\item $\Lambda$ is \emph{right-simple} is there exists $x_{R} \in \R$ for which each local maximum of $x|_{\Lambda}$ has $x$-value $x_{R}$.
\item $\Lambda$ is \emph{left-right-simple} if it is both left simple and right simple.
\ee
\end{defn}

Starting from a front projection, any $\Lambda$ can be made left-right-simple by applying type-II Reidemeister moves so that the front is plat and then applying a Lagrangian resolution as described in \cite{Ng:ComputableInvariants}.\footnote{A front diagram is \emph{plat} if all left-pointing cusps have the same $x$ coordinate and all right-pointing cusps have the same $x$ coordinate. Plat diagrams are of common use in the $LCH$ literature, cf. \cite{Sabloff:AugRuling}.} An example is worked out in Figures \ref{Fig:Simplification} and \ref{Fig:SimplificationLag}. Our proof of Theorem \ref{Thm:Main} principally relies on index calculations appearing in Theorem \ref{Thm:BranchIndex} (broadly applicable) and Theorem \ref{Thm:CritIndex} (specific to the left-right-simple setting).

\begin{figure}[h]
	\begin{overpic}[scale=.4]{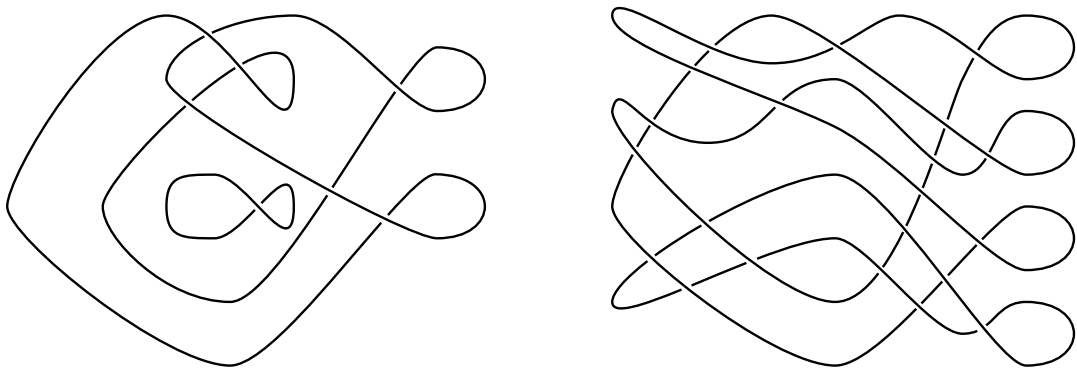}
	\end{overpic}
	\caption{Lagrangian resolutions of the fronts shown in Figure \ref{Fig:Simplification}.}
	\label{Fig:SimplificationLag}
\end{figure}

Only moduli spaces of Fredholm index $\leq 2$ holomorphic maps need to be considered for the development of typical $\SFT$ invariants. From the first statement of the above theorem we immediately obtain the following meta-result:

\begin{cor}\label{Cor:LSFT}
Any definition of $LSFT$ which counts rigid curves on (possibly multiple copies of)\footnote{See Section \ref{Sec:NCopy} for details on the $n$-copy construction.} $\Lambda$ is combinatorially computable by counting rigid disks with at most two positive punctures.
\end{cor}

\subsection{Curve counting}\label{Sec:CurveCountingIntro}

We describe our setup for curve counting and address the (im)possibility of counting $\chi(\dot{\Sigma}) < 1$ curves without a geometric constraint such as the left-right-simple condition.

\subsubsection{The obstruction function}

Following the analytical strategies of \cite{DR:Lifting, ENS:Orientations}, we obtain the following commutative diagram, which summarizes much of the content of Section \ref{Sec:ModSpaces}:
\begin{equation}\label{Eq:Lifting}
\begin{tikzcd}
\ModSpace^{\Lambda}/\R \arrow[r] \arrow[hookrightarrow, d] & 0 \arrow[hookrightarrow, d] \\
\ModSpace^{\lambda} \arrow[r, "\obstruction"] & H^{1}(\Sigma, \R)
\end{tikzcd}
\end{equation}

In words, the diagram says that we can identify the (reduced) moduli spaces $\ModSpace^{\Lambda}/\R$ of holomorphic curves in $\R \times \R \times \C$ with boundary on $\R \times \Lambda$ as the zero-locus, $\obstruction^{-1}(0)$, of an \emph{obstruction function}, $\obstruction$, from a moduli space $\ModSpace^{\lambda}$ of holomorphic curves in $\C$ with boundary on $\lambda$ to the first cohomology group of the domain, $\Sigma$. Here $\R$ acts on $\ModSpace^{\Lambda}$ by translating curves in the $s$-direction and we are allowing the topology of $\Sigma$ to vary across connected components of the moduli spaces. 

Of course $\obstruction$ always vanishes in the case $\Sigma = \disk$, yielding the known ``lifting'' results, \cite[Theorem 2.1]{DR:Lifting} and \cite[Theorem 7.7]{ENS:Orientations}. In the case $\Sigma \neq \disk$, we can view Equation \eqref{Eq:Lifting} as setting up a typically non-trivial obstruction bundle problem. For simplicity, we consider moduli spaces of parameterized maps. See Section \ref{Sec:Orbibundle} for further commentary on unparameterized curves.

\subsubsection{Curve counting difficulties}\label{Sec:CountDifficulties}

For $\Sigma$ with more general topology, Equation \eqref{Eq:Lifting} tells us that if $\ModSpace^{\Lambda}/\R$ has expected dimension $0$, then $\ModSpace^{\lambda}$ will have expected dimension
\begin{equation}\label{Eq:ExpdimIntro}
\expdim \ModSpace^{\lambda} = \dim H^{1}(\Sigma, \R) = 1 - \chi(\Sigma).
\end{equation}
For example, when $\Sigma$ is an annulus, then we want to study $\dim \ModSpace^{\lambda} = 1$ moduli spaces whose elements are not rigid curves. These are the dimensions required to count curves contributing to a $LSFT$ differential.

In such a situation, we would like to be able count points in $\mathcal{O}^{-1}(0)$ only using the data of the Lagrangian projection of $\Lambda$. The analysis of Section \ref{Sec:AnnuliExamples} demonstrations that this is not possible in general. However, the left-right-simple condition ensures that Equation \eqref{Eq:ExpdimIntro} cannot be achieved.

\subsection{Related work and additional context}

\subsubsection{Perturbed $\delbar$ equations}

Our approach to studying $\obstruction$ is via analysis of holomorphic curves which are perturbed in the sense of \cite{Abbas:JBook, ACH:PlanarWeinstein}. Contrasting with Doicu and Fuchs' \cite{DF:HCompactness}, our perturbation terms are unbounded, and this unboundedness is leveraged to count curves in special cases as previously mentioned. An ongoing program of Oh and Wang \cite{Oh:Instanton} describes solutions to generalization of these perturbed holomorphic curve equations.

\subsubsection{More on $LSFT$}

This article is partly motivated by recent work of Ekholm and Ng: In \cite{EkholmNg:LSFT}, a definition of Legendrian SFT is proposed in the context of knot contact homology, in which case counts of curves with $\chi(\Sigma) < 1$ may be reduced to counts of holomorphic disks via a recursion argument.

Although Theorem \ref{Thm:Main} may aid in computation, it does not \emph{apriori} assist in defining (some version of) $LSFT$, proving that such an algebraic object is invariant under Legendrian isotopy, or in addressing aspects of functoriality \cite{EHK:LagrangianCobordisms}. Analytical proofs of the invariance of $LSFT$ would require analysis of holomorphic curves on Lagrangian cobordisms. Likewise, combinatorial approaches to invariance should require violation of the left-right-simple condition. We do not address non-trivial Lagrangian cobordisms or algebraic aspects of $LSFT$ in this article.

\subsubsection{Heegaard-Floer invariants}

According to Lipshitz \cite{Lipshitz:Cylindrical}, the Heegaard-Floer invariants $\widehat{HF}(Y)$ \cite{OS:HF} of a smooth $3$-manifold, $Y$, count holomorphic curves in almost complex $4$-manifolds of the form $\R \times [0, 1] \times S$ for a surface $S$ with boundaries mapped to non-compact Lagrangian submanifolds. As in the case of $LSFT$ there are no restrictions on the topological types of the domains of curves. Therefore we view left-right-simple links as being analogous to Sarkar and Wang's \emph{nice Heegaard diagrams} \cite{SarkarWang}, and Theorem \ref{Thm:Main}(1) as being analogous to their \cite[Theorem 1.1]{SarkarWang} which reduces $\widehat{HF}(Y)$ curve counts to combinatorial counts of holomorphic disks in the Riemann surface $S$ with at most $2$ positive punctures. 

Our approach to analyzing $\ModSpace^{\Lambda}/\R$ described in Sections \ref{Sec:ModSpaces} and \ref{Sec:AnnuliExamples} can similarly be applied to $\widehat{HF}$ curve counts associated to non-nice Heegaard diagrams. This approach appears impractical for general computation but is useful in some restricted scenarios. The technique is used to count annuli in Ozv\'{a}th and Szab\'{o}'s original paper \cite[Section 9]{OS:HF} and in unpublished work of Pardon \cite{Pardon:HFLecture}.

\subsubsection{Replacing $\C$ with an arbitrary Liouville domain}

It would be interesting to know if an analogue of Theorem \ref{Thm:Main} exists for (perhaps topologically restricted classes of) Legendrian submanifolds in higher-dimensional contactizations of Liouville domains.\footnote{Our ambient space, $\Rthree$, may be viewed as the contactization of $T^{\ast}\R$.} The results of Section \ref{Sec:ModSpaces} readily generalize to holomorphic curves in symplectizations of contactizations of arbitrary Liouville domains in parallel with 
\be
\item Dimitroglou Rizell's \cite{DR:Lifting} generalizing \cite[Section 7]{ENS:Orientations} or
\item Colin, Honda, and Tian's construction of $\widehat{HF}$ for higher-dimensional manifolds \cite{CHT:HF}.
\ee

The story told in Section \ref{Sec:CurveCountingIntro} may be applied without modification to attempts at counting holomorphic curves ``by hand'' in any of the above generalized contexts just as it applies to $\widehat{HF}$. Hence a Sarkar-Wang-style result such as Theorem \ref{Thm:Main} may in general be necessary to eliminate rigid $\chi(\dot{\Sigma}) < 1$ curves, making combinatorially techniques such as Morse flow trees \cite{Ekholm:FlowTrees} applicable. Our proof of Theorem \ref{Thm:Main} relies on specifically low-dimensional methods -- combinatorial topology of curves on surfaces.

\subsection{Organization of this article}

In Section \ref{Sec:SimpleNotions} we cover generalities regarding Legendrian links in $\Rthree$. In Section \ref{Sec:ModSpaces} we define the moduli spaces of holomorphic curves relevant to this paper and elaborate on Equation \eqref{Eq:Lifting}. Section \ref{Sec:Index} describes general index formula for holomorphic maps to $\C$ with boundary on an immersed multi-curve. Examples are explored in which the map $\obstruction$ is used to combinatorially count holomorphic annuli in Section \ref{Sec:AnnuliExamples}. Finally, Section \ref{Sec:MainResults} is dedicated to the proof of Theorem \ref{Thm:Main}.

\subsection{Acknowledgments}

The main ideas for this article were conceived in conversations with Tobias Ekholm. We're also grateful for interesting discussions with Erkao Bao, Georgios Dimitroglou Rizell, Sam Lisi, and Lenhard Ng. We also thank Paolo Ghiggini for his lectures on $\widehat{HF}$ during his stay at Uppsala in 2021 and John Pardon for his correspondence. Finally, thanks to our anonymous referee for their detailed commentary which has helped to improve this article.

\section{Background and notions of simplicity for Legendrian links}\label{Sec:SimpleNotions}

In this section we describe Legendrian links in $\Rthree$ and explore various ways in which the geometry of such links may be restricted for the purposes of simplifying computations. In particular, the left-right-simple condition is described in Definition \ref{Def:LRS}.

\subsection{Basic notions}

Throughout, $\Lambda$ will denote a Legendrian link in $\Rthree$ and $\lambda$ will denote an immersed $1$-dimensional submanifold of $\C$ with only transverse self-intersections. Unless otherwise specified, we will take $\lambda = \pi_{\C}(\Lambda)$.

A \emph{chord of $\Lambda$} is a compact, connected $\partial_{t}$ trajectory in $\R^{3}$ which both begins and ends on $\Lambda$. Chords are in bijective correspondence with the double points of $\lambda$. The \emph{action}, $\action(\chord)$ of a chord $\chord = \chord(t)$ is defined
\begin{equation*}
\action(\chord) = \int_{\chord} dt.
\end{equation*}

Throughout we will use $\eta = \eta(T)$ to denote either an oriented path in $\lambda \subset \C$ which begins and ends at double points of $\lambda$. For such a path $\eta$ with domain $I \subset \R$, the \emph{rotation angle}, denoted $\theta(\eta) \in \R$, is defined as the integral of the Gauss map $G:\lambda \rightarrow \R / 2\pi\Z$ pulled back to an interval by $\eta$:
\begin{equation*}
\begin{tikzcd}
I \arrow[r, "\widetilde{G}"]\arrow[d, "\eta"] & \R \arrow[d]\\
\lambda \arrow[r, "G"] &  \R / 2\pi\Z
\end{tikzcd},\quad \theta(\eta) = \int_{I} \widetilde{G}dT.
\end{equation*}

\subsection{Restrictions on Lagrangian projections}

The following condition -- originally appearing in \cite{Avdek:Dynamics} -- will help us to simplify calculations of rotation angles paths in $\Lambda$. Similar conditions constraining the geometry of Legendrians near endpoints of chords appear in \cite{BEE:LegendrianSurgery, DR:Lifting, EES:LegendriansInR2nPlus1}.

\begin{defn}\label{Def:GoodPosition}
We say that an immersed multicurve $\lambda \subset \C$ is in \emph{good position} if each self-intersection $(x_{0},y_{0})\in\C$ is transverse and the there exists a neighborhood about the point within which
\be
\item one strand of $\lambda$ admits a parameterization satisfying $(x, y)(q) = (x_{0} + q, y_{0} - q)$ and
\item the other strand admits a parameterization satisfying $(x, y)(q) = (x_{0} + q, y_{0} + q)$.
\ee
We say that a Legendrian link $\Lambda \subset \Rthree$ is in \emph{good position} if $\pi_{\C}(\Lambda)$ is in good position and the first strand above corresponds to the one with greater $t$-value.
\end{defn}

Good position guarantees that the Gauss map of a parameterization of $\lambda$ evaluates to $\frac{3\pi}{4}$ or $\frac{7\pi}{4}$ near an over-crossing and to $\frac{\pi}{4}$ or $\frac{5\pi}{4}$ near an under-crossing. Likewise, the rotation angles of paths $\eta$ ending on double points of a multicurve in good position satisfy 
\begin{equation}\label{Eq:SimpleRotationAngle}
\theta(\eta) \in \frac{\pi}{2}\Z.
\end{equation}

\subsection{Lagrangian resolution}

The following proposition is a slight modification of \cite[Proposition 2.2]{Ng:ComputableInvariants} appearing in \cite{Avdek:Dynamics}:

\begin{prop}
Provided a front projection of a Legendrian link $\Lambda$, we may perform a Legendrian isotopy of so that it is in good position and so that away from double points the Lagrangian diagram is obtained by resolving singularities of the front as depicted in Figure \ref{Fig:LagrangianResolution}.
\end{prop}

\begin{figure}[h]\begin{overpic}[scale=.4]{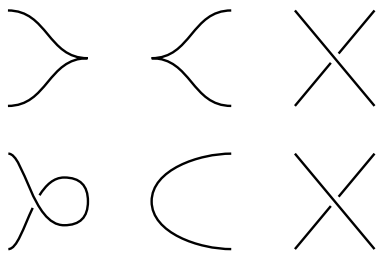}
	\end{overpic}
	\caption{The first row of subfigures shows segments of a Legendrian link appearing in the front projection. Directly below each subfigure is the corresponding local picture in the Lagrangian resolution.}
	\label{Fig:LagrangianResolution}
\end{figure}

We say that a Lagrangian projection obtained from a front as described in the above proposition is the \emph{Lagrangian resolution of the front}.

\subsection{Assumptions and conventions}\label{Sec:OvercrossingConvention}

Throughout the remainder of this article it will be assumed that every Lagrangian projection of a Legendrian knot is obtained by a Lagrangian resolution of a front diagram unless otherwise specified. Consequently, for the crossings appearing in both front and Lagrangian projections, we require that the ``north-west to south-east'' strand lies above the ``south-west to north-east'' strand of the link. With this simplification at hand, crossings in both the front- and Lagrangian projections will henceforth be drawn as ``X''s without ambiguity.

\subsection{Compatibility with the $n$-copy construction}\label{Sec:NCopy}

We note that left-right-simplicity is compatible with the satellite construction of \cite{NT:TorusLinks} when the satellite pattern is a positive braid. For the sake of brevity we only discuss the $n$-copy construction here, which is particularly important in the study of Legendrian links. For example, the $n$-copy is used to define product structures on linearized $LCH$ in \cite{BEE:Product, BC:Bilinearized} and $RSFT$ invariants of Legendrian links in \cite{Ekholm:Z2RSFT}.

For $n \in \N$, the \emph{$n$-copy} of Legendrian link, $\Lambda$, is defined as the Legendrian isotopy class of the link
\begin{equation*}
\Lambda^{n} = \sqcup_{k= 0}^{n-1} \Lambda_{k},\quad \Lambda_{k} = \Flow^{\frac{k}{n}\epsilon}_{\partial t}(\Lambda)
\end{equation*}
for $\epsilon > 0$ arbitrarily small, whence the Legendrian isotopy class of $\Lambda^{n}$ is independent of $\epsilon$.

\begin{figure}[h]
	\begin{overpic}[scale=.4]{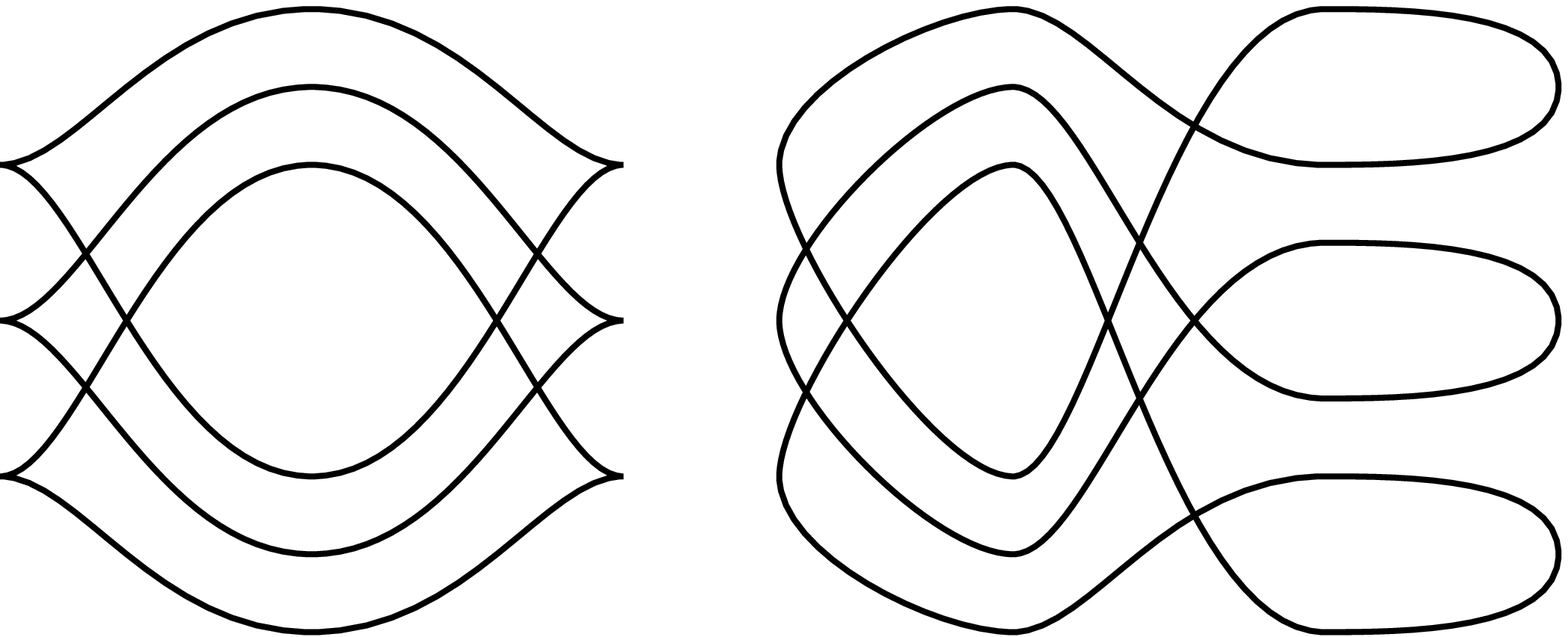}
	\end{overpic}
	\caption{A $3$-copy of a Legendrian unknot in the front and Lagrangian projections.}
	\label{Fig:UnknotThreeCopy}
\end{figure}

One obvious way to make $\Lambda^{n}$ left-right-simple is to perturb the front projection of $\Lambda$ to a plat closure as shown in Figure \ref{Fig:Simplification}, take the $n$-copy, and then apply the Lagrangian resolution. See Figure \ref{Fig:UnknotThreeCopy} for an example.

\begin{figure}[h]
\begin{overpic}[scale=.4]{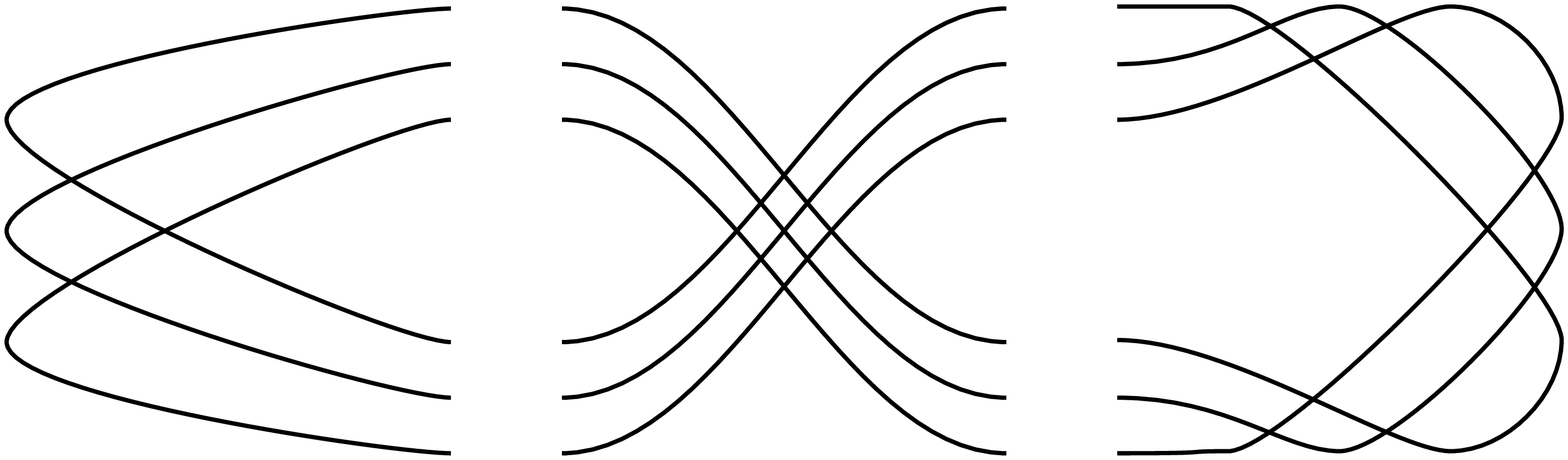}
\put(31, 28){$\Lambda_{2}$}
\put(31, 24){$\Lambda_{1}$}
\put(31, 20){$\Lambda_{0}$}
\put(66, 28){$\Lambda_{2}$}
\put(66, 24){$\Lambda_{1}$}
\put(66, 20){$\Lambda_{0}$}
\end{overpic}
\caption{Drawing an $n$-copy of a Lagrangian resolution in the Lagrangian projection for $n=3$. From left-to-right, the required modifications are depicted near the resolution of a left-pointing cusp, a crossing, and a left-pointing cusp.}
\label{Fig:NCopySchematic}
\end{figure}

Let $\Cfancy_{i, j}$ denote the chords of $\Lambda^{n}$ which start on $\Lambda_{i}$ and end on $\Lambda_{j}$. In \cite{BC:Bilinearized, BEE:Product, Ekholm:Z2RSFT} the authors apply Legendrian isotopies to $\Lambda^{n}$ to ensure chord genericity, requiring that the $\Lambda_{k}$ are close enough to the original $\Lambda$ so that
\be
\item when $i \geq j$, $\Cfancy_{i, j}$ is naturally identified with the chords of $\Lambda$.
\item when $i < j$, $\Cfancy_{i, j}$ is naturally identified with a union of the chords of $\Lambda$ together with critical points of a Morse function on $\Lambda$.
\ee

A method for producing a left-right-simple $n$-copy which meets these criteria from a left-right-simple $\Lambda$ is shown in Figure \ref{Fig:NCopySchematic}. This method does not rely on our assumption that the Lagrangian projection of $\Lambda$ is obtained by a Lagrangian resolution of a front. Note that patching the columns of Figure \ref{Fig:NCopySchematic} in the obvious fashion will produce a $3$-copy of the Legendrian unknot, Legendrian isotopic to the one shown in Figure \ref{Fig:UnknotThreeCopy}. The chords determined by the maxima and minima of $x|_{\Lambda}$ will correspond to the critical points of a Morse function on $\Lambda$ having $2\#(\Crit(x|_{\Lambda}))$ critical points.

\section{Moduli spaces and cohomological obstruction}\label{Sec:ModSpaces}

In this section we describe the moduli spaces of interest in this article. Our main result here is Theorem \ref{Thm:ObstructionCount} which justifies Equation \eqref{Eq:Lifting}. This section closes with a description of some technical subtleties.

As mentioned in the introduction, the definitions and results of this section may easily be extended to the case of Lagrangian cylinders over Legendrian submanifolds inside of symplectizations of contactizations of Liouville domains in the spirit of \cite[Theorem 2.1]{DR:Lifting}.

\subsection{Specification of $J$ on $\R^{4}$}\label{Sec:SpecifyJ}

We use $J_{0}$ to denote the standard complex structure on $\C$. As mentioned in the introduction, we always work with the standard complex structure $J = J_{0} \times J_{0}$ on $\R_{s} \times \R_{t} \times \C_{x, y} \simeq \C_{s, t} \times \C_{x, y}$ given by Equation \eqref{Eq:Jstd}. While $J$ does not preserve $\xi_{std}$, it is $\alpha_{std}$-tame in the sense of \cite{BH:ContactDefinition} and is so suitable for SFT computations. This standard $J$ is of common use in the Legendrian contact homology literature, cf. \cite{DR:Lifting}.

The results of this paper are equally valid using the $\alpha_{std}$-adapted almost complex structure $J'$ on $\R \times \R^{3}$ defined
\begin{equation*}
J'\partial_{s} = \partial_{t},\quad J' (v + ydx(v)\partial_{t}) = J_{0}v + ydx(J_{0}v)\partial_{t},\quad v \in T\C.
\end{equation*}
See, for example, \cite[Section 7]{ENS:Orientations} and \cite[Section 11.1]{Avdek:Dynamics}. Both $J$ and $J'$ are compatible with the symplectic form $d(e^{s}\alpha_{std})$, so that all Lagrangian submanifolds are totally real for either choice of almost complex structure.

\subsection{Riemann surfaces}

Throughout this article $\Sigma$ will denote a compact, connected, oriented surface with non-empty boundary. We write $H_{1}(\Sigma)$ ($H^{1}(\Sigma)$) for the associated first homology (cohomology) group with coefficients in $\R$. In this article we will frequently encounter $H^{1}$ elements as being determined by harmonic $1$-forms. It will be useful -- especially in Section \ref{Sec:AnnuliExamples} -- to think of $H^{1}$ as the dual space
\begin{equation*}
H^{1}(\Sigma) \simeq \Hom(H_{1}(\Sigma), \R)
\end{equation*}
whose elements are uniquely determined by their integrals over closed $1$-cycles. From this perspective, a choice of basis of homology of $\Sigma$ determines a dual basis of $H^{1}(\Sigma)$.

We write $\fatSigma = (\Sigma, j, \{ p_{i} \})$ for a compact, connected surface $\Sigma$ with labeled boundary components
\begin{equation*}
\partial \Sigma_{k},\quad k=1,\dots,\#(\partial \Sigma),
\end{equation*}
complex structure $j$, and a non-empty collection of marked points $p_{i}$ contained in $\partial \Sigma$. We say that such a triple $(\Sigma, j, \{ p_{i} \})$ is a \emph{decorated Riemann surface}. 

We say that two decorated Riemann surfaces $\fatSigma = (\Sigma, j, \{ p_{i} \})$ and $\widetilde{\fatSigma} = (\widetilde{\Sigma}, \widetilde{j}, \{ \widetilde{p_{i}} \})$ are \emph{isomorphic} if $\#(\partial \Sigma_{k}) = \#(\partial \widetilde{\Sigma}_{k})$, $\#(p_{i}) = \#(\widetilde{p}_{i})$, and there is a $(j, \widetilde{j})$ holomorphic diffeomorphism $\Sigma \rightarrow \widetilde{\Sigma}$ mapping each $\partial \Sigma_{k}$ to $\partial \widetilde{\Sigma}_{k}$ and each $p_{i}$ to $\widetilde{p}_{i}$. We say that $\fatSigma$ and $\widetilde{\fatSigma}$ are \emph{isotopic} if such a diffeomorphism exists which is isotopic to the identity in $\Diff(\Sigma)$.

We write $\dot{\Sigma}$ for $\Sigma$ with all of its marked points removed, $\partial \dot{\Sigma}_{k}$ for the $k$th boundary component of $\Sigma$ with the $p_{i}$ removed, and $\partial \dot{\Sigma}$ for the union of the $\partial \dot{\Sigma}_{k}$. The following lemma summarizes some results of Sections 3 and 4 of \cite{Liu:Moduli}:

\begin{lemma}\label{Lemma:SurfaceModSpaceDim}
We say that the triple $(\Sigma, j, \{ p_{i}\} \} )$ is \emph{stable} if $\Aut(\dot{\Sigma}, j)$ is finite. This condition is equivalent to the condition that 
\begin{equation*}
2\chi(\Sigma) - \#(p_{i})< 0.
\end{equation*}
Assuming stability, the moduli space of isotopy classes of decorated Riemann surfaces $\boldsymbol{\Sigma}$ is a manifold of dimension
\begin{equation*}
\#(p_{i}) - 3\chi(\Sigma) = \#(p_{i}) + 3\#(\partial \Sigma) + 6g(\Sigma) - 6.
\end{equation*}
\end{lemma}

The only non-stable curves we consider are the $2$-disk, $\disk$, with one or two boundary punctures, which have automorphism groups of dimensions $2$ and $1$ respectively.

We defer to \cite{Liu:Moduli} for descriptions of the compactifications of moduli spaces of decorated Riemann surfaces. In particular, we assume familiarity with nodal curves which will appear in the proof of Theorem \ref{Thm:Main}. It is of note that we have not included interior punctures or surfaces without boundary marked points. 

Because we will be interested in holomorphic maps with Lagrangian boundary in a Euclidian space, all interior punctures of maps under consideration will be removeble. Limits of holomorphic curves may of course develop interior punctures in their domains as the boundary component of a curve may shrink to a point. We will see non-trivial examples of this phenomenon in Section \ref{Sec:AnnuliExamples}. Because the Lagrangians of interest will be cylinders over Legendrian links, there will be no holomorphic maps without at least one boundary puncture. 

\subsection{Moduli spaces of holomorphic $u$}

Let $\lambda$ be a compact, immersed multi-curve in $\C$ in good position. We consider $(j, J_{0})$-holomorphic maps $u: \dot{\Sigma} \rightarrow \C$ which satisfy the following conditions:
\be 
\item $u$ satisfies the Lagrangian boundary condition, $u(\partial \dot{\Sigma}) \subset \lambda$.
\item The \emph{energy}, $\energy(u)$, of $u$ is finite, where it is defined
\begin{equation*}
\energy(u) = \int_{\Sigma} u^{\ast}(dx \wedge dy).
\end{equation*}
\item All boundary punctures are asymptotic to self-intersection of $\lambda$.
\ee

Two such maps will be considered equivalent if they differ by a reparameterization of the domain which is an isotopy of decorated Riemann surfaces. Therefore our moduli spaces will be modeled on Teichm\"{u}ller space rather than on the Deligne-Mumford space of curves, typically used in SFT. See Section \ref{Sec:Orbibundle} for further commentary.

Provided such a map $u$ we write $\ModSpace^{\lambda}_{u}$ for path connected component of the space of such $(j',J_{0})$-holomorphic maps $u': \dot{\Sigma} \rightarrow \C$ containing $u$, where $j'$ a complex structure on $\dot{\Sigma}$, modulo biholomorphic reparameterization. Without a map $u$ specified, we will write $\ModSpace^{\lambda}$ for the associated moduli space.

According to our convention that holomorphic maps are parameterized up to isotopy together with fact that isotopy acts trivially on cohomology, $H^{1}(\Sigma)$ forms a trivializable vector bundle over each $\ModSpace^{\lambda}_{u}$. Such a choice of trivialization is equivalent to a choice of basis of $H^{1}(\Sigma)$.

\subsection{Moduli spaces of Dirichlet problems}

Now suppose that $\Lambda$ is a Legendrian link in $\R^{3}$ and write $\lambda = \pi_{\C}(\Lambda)$ for its projection to the $xy$-plane, $\C$. We assume that $\Lambda$ is in good position.

Consider the moduli space $\ModSpace^{\Dfancy}$ whose elements consist of pairs $( t^{\partial}, u)$ of holomorphic maps $u \in \ModSpace^{\lambda}$ together with continuous functions $t^{\partial}: \partial \dot{\Sigma} \rightarrow \R$ such that
\begin{equation*}
(t^{\partial}(z), u(z)) \in \Lambda \subset \R^{3}, \quad z \in \partial \dot{\Sigma}.
\end{equation*}
We declare that $(t^{\partial}_{i}, u_{i})$ for $i=1, 2$ are equivalent if they differ by isotopic reparameterizations of the domain. We call $\ModSpace^{\Dfancy}$ the \emph{Dirichlet moduli space}. For a particular choice of $(t^{\partial}, u)$ the path-connected component of $\ModSpace^{\Dfancy}$ containing the pair $(t^{\partial}, u)$ will be denoted $\ModSpace^{\Dfancy}_{(t^{\partial}, u)}$.

\begin{prop}\label{Prop:DirichletProjection}
The map 
\begin{equation*}
\pi_{\Dfancy}: \ModSpace^{\Dfancy} \rightarrow \ModSpace^{\lambda},\quad \pi_{\Dfancy}(t^{\partial}, u) = u
\end{equation*}
is a homeomorphism away from constant maps.
\end{prop}

\begin{proof}
Because $\Lambda$ is a submanifold of $\R^{3}$ and the function $t^{\partial}$ is required to be continuous, the restriction of $t^{\partial}$ to any connected component $\partial \dot{\Sigma}$ is determined by its value at a given point which is not mapped to a double point of $\lambda$. Because we have assumed that $u$ is non-constant, the restriction of $u$ to each component of $\partial \dot{\Sigma}$ must be non-constant. Hence the image of each such connected component must contain a point $z$ for which $u(z)$ is not a double point of $\lambda$. For such $z$, there is a unique choice of $t^{\partial}(z)$. Thus $t^{\partial}$ is uniquely determined by the map $u$.
\end{proof}

\subsection{Moduli spaces of holomorphic $U$}\label{Sec:HoloUDefn}

We now consider $(j, J)$-holomorphic maps 
\begin{equation*}
U = (s, t, u): \Sigma \rightarrow \R \times \R \times \C
\end{equation*}
subject to the following conditions:
\be
\item $U$ has Lagrangian boundary, $U(\partial \dot{\Sigma}) \subset \R \times \Lambda$.
\item The \emph{$d\alpha$-energy}, $\energy_{d\alpha}(U)$, of $U$ is finite, where it is defined \begin{equation*}
\energy_{d\alpha}(U) = \int_{\Sigma}U^{\ast}d\alpha_{std}.
\end{equation*}
\item The \emph{Hofer energy}, $\energy_{H}(U)$, of $U$ is finite, where it is defined
\begin{equation*}
\energy_{H}(U) = \sup_{\rho \in \mathcal{F}} \int_{\Sigma} U^{\ast} (\rho(s) ds \wedge \alpha),\quad \mathcal{F} = \{ \rho: \R \rightarrow [0, 1]:\ \frac{d \rho}{d s} \geq 0\}.
\end{equation*}
\item All boundary punctures asymptotic to chords of $\Lambda$. That is, for each boundary puncture $p \in \partial \Sigma$ there is a holomorphic coordinate system of the form $[C_{0}, \infty)_{C} \times [0, \pi] \subset \dot{\Sigma}$ identifying the point at $C  = \infty$ with $p$ so that $\lim_{z \rightarrow p} s(z) = \pm \infty$ and we have convergence of the paths $\lim_{C\rightarrow \infty}(t(C, \ast), u(C, \ast))$ to a chord of $\Lambda$ with $\pm$-orientation.
\ee

As $\pi_{\C}$ is $(j, J_{0})$-holomorphic and $d\alpha_{std} = dx \wedge dy$, the $d\alpha$-energy is determined by $u = \pi_{\C}\circ U$. Stokes' theorem also tells us that the energy of $U$ is determined by the chords $\{ \chord_{j}^{\pm} \}_{j=1}^{m^{\pm}}$ to which it is positively and negatively asymptotic:
\begin{equation}\label{Eq:ActionEnergy}
\energy_{d\alpha}(U) = \energy(u) = \sum_{1}^{m^{+}} \action(\chord_{j}^{+}) - \sum_{1}^{m^{-}} \action(\chord_{j}^{-}).
\end{equation}

Provided such a map $U$ we write $\ModSpace^{\Lambda}_{U}$ for the path connected component of the space of such $(j',J)$-holomorphic maps $U': \dot{\Sigma} \rightarrow \R \times \C \times \R$ containing $U$, where $j'$ a complex structure on $\dot{\Sigma}$, modulo isotopic reparameterization. This moduli space admits an $\R$-action by translation in the $s$-direction, $(s_{0}, U) \mapsto U + (s_{0}, 0, 0, 0)$, and we write $\ModSpace^{\Lambda}_{U}/\R$ for the quotient. Disregarding a particular $U$, the moduli space and its $\R$-quotient will be denoted $\ModSpace^{\Lambda}$ and $\ModSpace^{\Lambda}/\R$, respectively.

\subsection{Cohomological obstructions}\label{Sec:Dirichlet}

For a function $f: \dot{\Sigma} \rightarrow \R$ defined on a subset $\dot{\Sigma} \subset \Sigma$ of a compact Riemann surface $(\Sigma, j)$, the Laplacian $\Delta f$ is defined\footnote{What we're calling the Laplacian -- taking values in $\Omega^{2}(\dot{\Sigma})$ rather than the usual $\Omega^{0}(\dot{\Sigma}) = \Cinfty(\dot{\Sigma})$ -- is typically called the Levi form. See \cite{OS:SurgeryBook, SteinToWeinstein}.}
\begin{equation*}
\Delta f = -d (df \circ j) \in \Omega^{2}(\dot{\Sigma})
\end{equation*}
which in a local holomorphic coordinate system $x, y$ takes the usual form 
\begin{equation*}
\Delta f = (\frac{\partial^{2} f}{\partial x^{2}} + \frac{\partial^{2} f}{\partial y^{2}})dx\wedge dy.
\end{equation*}

\begin{lemma}
Let $\Lambda \subset \Rthree$ be a Legendrian link. Then a map 
\begin{equation*}
U = (s, t, u): \dot{\Sigma} \rightarrow \R \times \R \times \C, \quad U(\partial \dot{\Sigma}) \subset \R \times \Lambda
\end{equation*}
is $(j, J)$-holomorphic if and only if the following are satisfied,
\begin{equation*}
\delbar_{J_{0}} u = 0,\quad ds = dt\circ j,
\end{equation*}
in which case $\Delta t = 0$.
\end{lemma}

This is a straightforward computation. To address Question \ref{Q:Main}, we want to understand to what extent elements of $\ModSpace^{\Lambda}$ are determined by $(t, u) = \pi_{\R^{3}} \circ U$.  With this line of inquiry in mind -- together with the fact that harmonic functions are determined by their boundary values -- we define a map
\begin{equation*}
\pi_{\Lambda}: \ModSpace^{\Lambda}/\R \rightarrow \ModSpace^{\Dfancy},\quad \pi_{\Lambda}([(s, t, u)]) = (t|_{\partial \dot{\Sigma}}, u).
\end{equation*}
The content of \cite[Theorem 2.1]{DR:Lifting} and \cite[Theorem 7.7]{ENS:Orientations} is that $\pi_{\Lambda}$ is a homeomorphism when $\Sigma = \disk$. We follow their strategy.

Provided $(t^{\partial}, u)$, the existence of a function $t \in \Cinfty(\dot{\Sigma})$ solving the Dirichlet problem
\begin{equation}\label{Eq:Dirichlet}
t|_{\partial \dot{\Sigma}} = t^{\partial},\quad \Delta t = 0
\end{equation}
may be established using classical techniques, such as Perron's method as described in \cite[Section 6.4.2]{Ahlfors}\footnote{While the solution to the Dirichlet problem in \cite{Ahlfors} is worked out for bounded regions $\Omega \subset \C$ with non-singular $\partial \Omega$ and bounded, piece-wise continuous $t^{\partial}$, the proof relies only on analysis of functions on holomorphically embedded disks $\subset \Omega$, and so generalizes to arbitrary compact Riemann surfaces with only notational modification.}. In contrast with \cite{DR:Lifting, ENS:Orientations}, the required exactness of $dt \circ j$ is not automatically satisfied when $\chi(\Sigma) < 1$. To each Dirichlet problem $(t^{\partial}, u) \in \ModSpace^{\Dfancy}$ we define the \emph{obstruction class}
\begin{equation}\label{Eq:ObstructionDef}
\obstruction: \ModSpace^{\Dfancy} \rightarrow H^{1}(\Sigma, \R),\quad \obstruction((t^{\partial}, u)) = [dt \circ j].
\end{equation}
The fact that $\Delta t = 0$ is equivalent to $dt \circ j \in \Omega^{1}(\Sigma)$ being closed establishes that $\obstruction((t^{\partial}, u))$ is indeed a cohomological cycle. As we have declared that holomorphic maps are equivalent only if they differ by an isotopy of $\fatSigma$, and any such isotopy will act trivially on $H^{1}(\Sigma, \R)$, $\obstruction$ is well defined.

\begin{thm}\label{Thm:ObstructionCount}
The map $\pi_{\Lambda}$ is an injection whose image is $\obstruction^{-1}(0)$.
\end{thm}

\begin{proof}
Continuting the above discussion, we show that each $(t^{\partial}, u) \in \obstruction^{-1}(0)$ admits a lift to $\ModSpace^{\Lambda}$. Provided $t \in \Cinfty(\dot{\Sigma})$ solving Equation \eqref{Eq:Dirichlet}, the definition of $\obstruction$ then entails that there exists $s:\dot{\Sigma} \rightarrow \R$ solving $ds = dt \circ j$. We must verify that $U = (s, t, u)$ meets the criteria for inclusion in $\ModSpace^{\Lambda}$ as described in Section \ref{Sec:HoloUDefn}. 

The Lagrangian boundary condition for $\R \times \Lambda$ is satisfied by definition of the space $\ModSpace^{\Dfancy}$. The finiteness of $\energy_{d\alpha}(U)$ follows from Equation \eqref{Eq:ActionEnergy} together with the presumed finiteness of $\energy(u)$. For the finiteness of the Hofer energy, we appeal to \cite[Lemma B.3]{Ekholm:Z2RSFT}, which bounds $\energy_{H}(U)$ in terms of the actions of the positive punctures of $U$. 

To complete our proof that $(s, t, u) \in \ModSpace^{\Lambda}$, we must show that the boundary punctures of $\dot{\Sigma}$ are positively and negatively asymptotic to chords of $\Lambda$. The case $\Sigma = \disk$ is worked out in the proof of \cite[Theorem 7.7]{ENS:Orientations}. As the analysis of \cite{ENS:Orientations} is carried out in a strip-like neighborhood of each boundary puncture whence it is reduced to the case considered in \cite{RS:Strips}, it translates without modification to the case of a boundary puncture in a general decorated Riemann surface.

We have so far established that $\obstruction^{-1}(0) \subset \pi_{\Lambda}(\ModSpace^{\Lambda}/\R)$. The opposite inclusion $\pi_{\Lambda}(\ModSpace^{\Lambda}/\R) \subset \obstruction^{-1}(0)$ as well as the uniqueness of lifts follow from the fact that $(s, t, u) \in \ModSpace^{\Lambda}$ is uniquely determined -- up to addition of constant functions of the form $(s_{0}, 0, 0)$ -- by $(t|_{\partial \dot{\Sigma}}, u)$. For $t$, this is a consequence of the maximum principle for harmonic functions. For $s$, this follows from the fact that if $ds' = ds$ then $s' = s + s_{0}$ for some $s_{0} \in \R$.
\end{proof}

\subsection{Some technical nuances}\label{Sec:TechnicalIssues}

\subsubsection{When $u$ is constant}

The map $\pi_{\Dfancy}$ will typically be many-to-one along the inverse image of constant maps $u$ whose image is a self-intersection of $\lambda$. Supposing that $(x_{0}, y_{0}) \in \C$ is such a self-intersection corresponding to a chord $c$ of $\Lambda$, which for notational simplicity, we assume starts at $t = 0$. Supposing also that $\fatSigma$ is a decorated Riemann surface, there is only a single holomorphic map to $\C$ with image $(x_{0}, y_{0})$, but $2^{N}$ holomorphic maps in $\pi_{\Dfancy}^{-1}((x_{0}, y_{0}))$ where $N$ is the number of connected components of $\partial \dot{\Sigma}$. Indeed, each factor of $2$ accounts for a choice of sending each connected component of $\partial \dot{\Sigma}$ to either the starting or ending point of the chord lying over $(x_{0}, y_{0})$.

\begin{comment}
The disk with two boundary punctures $\Sigma = \disk$, $\{ p_{i} \} = \{ \pm 1\}$, is conformally equivalent to an infinite strip $\R_{s} \times i[0, \action(\chord)]_{t}$. The Dirichlet problem $t^{\partial}|_{\R \times \{ i\action{\chord}\}} = \action(\chord), t^{\partial}|_{\R \times \{ 0\}} = 0$, $u = (x_{0}, y_{0})$ will determine the trivial strip $(s, t) \mapsto (s, t, 0)$ over $\chord$.
\end{comment}

\begin{figure}[h]
\begin{overpic}[scale=.4]{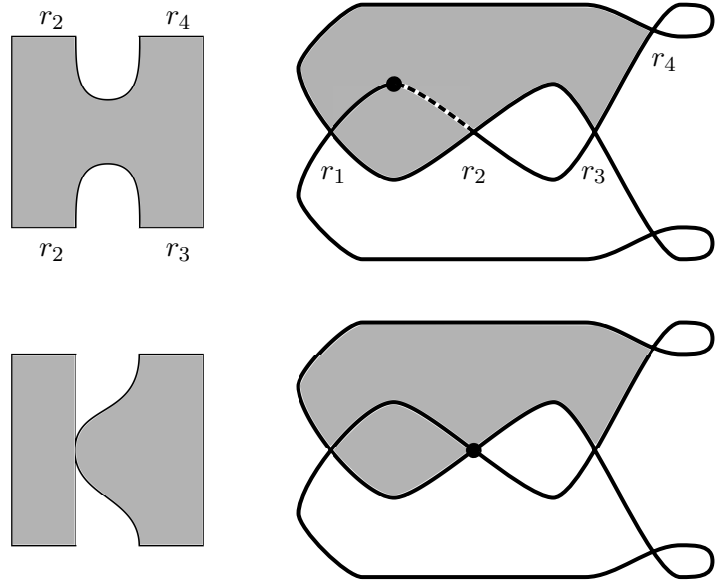}
\put(4, 79){$\chord_{2}$}
\put(22, 79){$\chord_{4}$}
\put(4, 46){$\chord_{2}$}
\put(22, 46){$\chord_{3}$}

\put(44, 57){$\chord_{1}$}
\put(64, 57){$\chord_{2}$}
\put(81, 57){$\chord_{3}$}
\put(91, 73){$\chord_{4}$}
\end{overpic}
\caption{A $1$-parameter family of holomorphic curves in $\R \times \R^{3}$ with boundary on $\R \times \Lambda$, where $\Lambda$ is a trefoil. The images of the $u$ are shown in the right-hand column with boundary critical points shown as a thick dot. As the image of the boundary critical point tends towards $\pi_{\C}(\chord_{2})$ in $\C$, a trivial strip over $\chord_{2}$ bubbles off.}
\label{Fig:TrefoilCLBreaking}
\end{figure}

As described in \cite{CL:SFTStringTop}, curves of positive area with multiple positive punctures may degenerate into nodal configuration containing trivial strips. An example is shown in Figure \ref{Fig:TrefoilCLBreaking}. This type of degeneration necessitates that the appearance of the string topological correction term in the differential of Ng's Legendrian RSFT \cite{Ng:RSFT}.

\subsubsection{Compactness}

The preceding example indicates that care must be taken when attempting to extend the results of this section pertaining to the moduli spaces $\ModSpace^{\ast}$ to their corresponding compactifications, $\overline{\ModSpace^{\ast}}$.

For another example: We will see in Section \ref{Sec:CompactifiedDirichlet} that the obvious extension of the map $\pi_{\Dfancy}$ to the associated compactified moduli spaces $\overline{\ModSpace^{\Dfancy}} \rightarrow \overline{\ModSpace^{\lambda}}$ is not a homeomorphism in general.

More foundational issues arise when studying the interplay between Dirichlet problems and limits of decorated Riemann surfaces: The classical example
\begin{equation*}
\Int(\Sigma) = \{ |z| \in (0, 1) \} \subset \C,\quad t = \begin{cases}
0 & |z| = 1\\
1 & z = 0
\end{cases}
\end{equation*}
indicates that the Dirichlet problem of Equation \eqref{Eq:Dirichlet} cannot be solved in general over surfaces with singular boundary, which may be realized as limits of non-singular decorated Riemann surfaces. Thus the map $\obstruction$ -- as defined using Dirichlet solutions in Equation \eqref{Eq:ObstructionDef} -- does not generally admit a continuous extension to $\partial \overline{\ModSpace^{\Dfancy}}$. We will leverage this lack of continuity to draw pictures of moduli cycles of annuli in Section \ref{Sec:AnnuliExamples}.

\subsubsection{Orbibundles over moduli spaces of unparameterized maps}\label{Sec:Orbibundle}

If we had defined our moduli spaces by declaring holomorphic maps to be equivalent if they differ by isomorphism -- rather than isotopy -- of decorated Riemann surfaces, then $\obstruction$ would be ill-defined due to the possibility for $\Aut(\fatSigma)$ to act non-trivially on $H^{1}(\Sigma)$. Then $H^{1}(\Sigma)$ would form an orbibundle over the $\ModSpace^{\Dfancy}$ moduli spaces of which $\obstruction$ would be a lifted multisection. See, for example, \cite{FT:Kuranishi}.

We have chosen to work with isotopies of $\fatSigma$ to avoid these complexities. Theorem \ref{Thm:Main} indicates that for $LSFT$ curves $U$ with $\ind(U) \leq 2$ there is no difference between moduli spaces of isotopic and isomorphic curves when the left-right-simple condition is in effect.

\subsubsection{Perturbed holomorphic maps}

For the purposes of either topological applications or perturbing the function $\obstruction$, the reader may be interested in solving equations of the form
\begin{equation}\label{Eq:DbarPert}
\begin{gathered}
U = (s, t, u): \Sigma \rightarrow \R \times \R \times \C, \quad \delbar u = 0,\quad ds + dt \circ j = \gamma,\\
\gamma \in \Omega^{1}(\Sigma),\quad d\gamma = d(\gamma \circ j) = 0.
\end{gathered}
\end{equation}
We say that solutions of the above equations satisfying the conditions of Section \ref{Sec:HoloUDefn} are \emph{harmonically perturbed holomorphic curves} as studied in \cite{Abbas:JBook, ACH:PlanarWeinstein, DF:HCompactness}.

\begin{prop}
For each $(t^{\partial}, u) \in \ModSpace^{\lambda}$ there exists a harmonically perturbed holomorphic curve $U$ with perturbation term $\gamma \in \Omega^{1}$ satisfying
\begin{equation*}
\pi_{\C}(U) = u, \quad U(\partial \dot{\Sigma}) \subset \R \times \Lambda, \quad  [\gamma] = \obstruction((t^{\partial}, u)) \in H^{1}(\Sigma)
\end{equation*}
\end{prop}

Observe that the asymptotic convergence of boundary punctures to Reeb chords here is analogous to the asymptotic convergence of interior punctures to closed Reeb orbits in the $\partial \Sigma = \emptyset$ case described by Abbas in \cite[Proposition 1.5]{Abbas:JBook}.

\begin{proof}
We proceed as in the proof of Theorem \ref{Thm:ObstructionCount} by first finding a function $t$ solving the Dirichlet problem of Equation \eqref{Eq:Dirichlet}. Now we construct the harmonic form $\gamma \in \Omega^{1}(\Sigma)$ with the appropriate cohomology class.

Let $(\widetilde{\Sigma}, \widetilde{j})$ be a closed Riemann surface with a holomorphic embedding $i: \Sigma \rightarrow \widetilde{\Sigma}$ for which $i^{\ast}: H^{1}(\widetilde{\Sigma}) \rightarrow H^{1}(\Sigma)$ is a surjection. The doubling constructions described in \cite{Liu:Moduli} provide some options of such $(\widetilde{\Sigma}, \widetilde{j})$. Applying a Hodge decomposition to $\widetilde{\Sigma}$, we can find a harmonic $\widetilde{\gamma} \in \Omega^{1}(\widetilde{\Sigma})$ for which $[i^{\ast}\widetilde{\gamma}] = \obstruction = [dt \circ j]$ and then define $\gamma = i^{\ast}\widetilde{\gamma}$. It follows that $\gamma$ is harmonic and $\gamma - dt \circ j$ is cohomologically trivial, so that there exists an $s$ solving Equation \eqref{Eq:DbarPert}. Moreover, the closure and co-closure of $\gamma$ ensure that $s$ is harmonic as well.

Having defined the function $s$ and so the map $U = (s, t, u)$, we address the $d\alpha$- and Hofer-energies. By the fact that $u$ is holomorphic, Equation \eqref{Eq:ActionEnergy} still applies so that $\energy_{d\alpha}(U) < \infty$. Again using $\delbar u = 0$ -- implying $u^{\ast}d\alpha_{std} \geq 0$ point-wise on $\dot{\Sigma}$ -- the proof of \cite[Theorem B.3]{Ekholm:Z2RSFT} is still valid, implying $\energy_{H}(U) < 0$.

As for asymptotics, exponential converge in $u$ of boundary punctures towards a self-intersection of $\lambda$ follows from \cite{RS:Strips} as in the proof of Theorem \ref{Thm:ObstructionCount}. To complete the proof, we will establish asymptotic convergence towards a chord in the $s$ and $t$ directions. 

Near each boundary puncture of $\Sigma$, we can use one of two holomorphic models. First we can model the puncture on the half-disk $\disk \cap \{ y \leq 0\} \subset \C$ with the boundary marked point corresponding to $z = 0$. Next, we can use the strip model, $[0, \infty) \times i[0, \pi] \subset \C$ with the point at infinity, $x \rightarrow \infty$ corresponding to the marked point. To get from the strip model to the half-disk model, we simply apply the transformation $z = e^{-\zeta}$.

In either model $\gamma$ is exact and $\gamma = df$, so that $s - f$ is holomorphic. Applying the arguments of \cite{DR:Lifting, ENS:Orientations} as in Theorem \ref{Thm:ObstructionCount}, we see that $(s - f, t)$ -- which is holomorphic -- satisfies the usual exponential convergence estimate in the half-infinite strip model. We may assume the $f = 0$ at our boundary  puncture, so that a Taylor expansion for $f$ in the half disk model yields $|f| = \bigO(|z|) = \bigO(e^{|\zeta|})$. Therefore $(s, t) = (s -f, t) + (f, 0)$ will exponentially converge to the trivial strip over some chord as desired.
\end{proof}

\section{Combinatorial index computations}\label{Sec:Index}

Here we describe two ways of computing indices of holomorphic maps in $\ModSpace^{\lambda}$ and $\ModSpace^{\Lambda}$: First we state the usual index formulae using Maslov indices. Second, we compute indices by counting critical points of holomorphic maps from a surface to the plane in Theorem \ref{Thm:BranchIndex}.

Throughout this section, the good position condition is assumed to hold. This is primarily to simplify the calculations of rotation angles as described in Equation \eqref{Eq:SimpleRotationAngle}. We do not assume that $\lambda$ or $\Lambda$ is left-right-simple.

\subsection{Index calculations from Maslov numbers}

\begin{thm}\label{Thm:IndMaslovTwoD}
The expected dimension of the space $\ModSpace^{\lambda}_{u}$ is given by the index formula
\begin{equation*}
\ind(u) = \sum_{k=1}^{\#(\partial \Sigma)} \Maslov(u|_{\partial \Sigma_{k}}) - 2\chi(\Sigma) + \#(p_{i}).
\end{equation*}
\end{thm}

This is a special case of the index formula computation of \cite[Theorem A.1]{CEL:Switching} in the case $n=1$ with only Lagrangian intersection boundary punctures using the (unique up to homotopy) framing of $T\C$ which extends over all of $\C$. The Maslov index over a boundary component $\partial \Sigma_{k}$ with punctures $p_{k, i}$ and oriented, connected components of $\partial \dot{\Sigma}_{k}$ denoted $\eta_{k, i}$ may be computed
\begin{equation}\label{Eq:MaslovRotation}
\Maslov(u|_{\partial \Sigma_{k}}) = \frac{1}{\pi}\sum \theta(u|_{\eta_{k, i}}) - \half\#(p_{k, i}).
\end{equation}
Here the $-\half$ contribution at each boundary puncture comes from the fact that a negative rotations of a Lagrangian plane in $\C^{n}$ is applied along each Lagrangian intersection (or Reeb chord) in order to assign a closed loop of Lagrangian subspaces to each $\partial \Sigma_{k}$. See \cite{CEL:Switching} for further context and details.

\begin{defn}
We say that such a map $u$ is \emph{rigid} if $\ind(u) = 0$. 
\end{defn}

We expect that rigid maps would be isolated in the sense that $\ModSpace_{u} = \{ [u] \}$ where $[u]$ is the class of maps which agree with $u$ after biholomorphic isotopy.

\begin{thm}\label{Thm:IndMaslovThreeD}
The expected dimension of the moduli space $\ModSpace^{\Lambda}_{U}$ is given by the index formula
\begin{equation*}
\ind(U) = \sum_{k=1}^{\#(\partial \Sigma)} \Maslov(U|_{\partial \Sigma_{k}}) - \chi(\Sigma) + \#(p_{i})
\end{equation*}
\end{thm}

This is also a special case of \cite[Theorem A.1]{CEL:Switching}, this time in the case $n=2$ with only Reeb chord boundary punctures. Combining Theorems \ref{Thm:IndMaslovTwoD} and \ref{Thm:IndMaslovThreeD} we have
\begin{equation}\label{Eq:IndProjComparison}
\ind(U) = \ind(u) + \chi(\Sigma)
\end{equation}
which follows from the fact that 
\begin{equation*}
\Maslov(U|_{\partial \Sigma_{k}}) = \Maslov(u|_{\partial \Sigma_{k}}).
\end{equation*}
Indeed, $T(\R \times \Lambda)$ splits as $\R \partial_{s} \oplus T\lambda$ and the Maslov number of $\R \partial_{s}$ is trivial over any loop, so the formula follows from the additivity of the Maslov number under direct sums \cite[Theorem 2.29]{MS:SymplecticIntro}.

\begin{defn}
We say that such a map $U$ is \emph{rigid} if $\ind(U) = 1$.
\end{defn}

We expect that a rigid map $U$ would be isolated in the sense that $\ModSpace^{\Lambda}_{U} \simeq \R$, parameterized by maps which are $\R$-translates (in the $s$ direction) of maps biholomorphically isotopic to $U$. Equation \eqref{Eq:IndProjComparison} indicates that in order to study $d$-dimensional moduli spaces $\ModSpace_{U}^{\Lambda}$ we must understand the obstruction class $\obstruction$ over moduli spaces $\ModSpace_{u}^{\lambda}$ of dimension
\begin{equation}\label{Eq:DimensionComparison}
\dim(\ModSpace_{u}^{\lambda}) = d - \chi(\Sigma) = d + \dim H^{1}(\Sigma) - 1.
\end{equation}

\begin{rmk}
As with the content of Section \ref{Sec:ModSpaces}, Equations \eqref{Eq:IndProjComparison} and \eqref{Eq:DimensionComparison} readily translate to the contexts of \cite{CHT:HF, DR:Lifting, Lipshitz:Cylindrical}.
\end{rmk}

\subsection{Local branching of holomorphic maps}

For a holomorphic map $u \in \ModSpace_{u}$ we write $\Crit^{\ast}(u) \subset \Sigma$ for the sets of critical points
\begin{equation*}
\Crit(u) = \{z\in \dot{\Sigma},\ Tu(z) = 0 \}, \quad \Crit^{\Int}(u) = \Crit(u)\cap \Int(\dot{\Sigma}),\quad \Crit^{\partial}(u) = \Crit(u) \cap \partial \dot{\Sigma}.
\end{equation*}
We only consider non-constant $u$ with $\Sigma$ compact so that $\Crit(u)$ is a finite subset of $\Sigma$, cf. \cite{Ahlfors, MS:SymplecticIntro}.

\subsubsection{Orders of critical points}

For $z_{0} \in \dot{\Sigma}$ we define the \emph{order of $u$ at $z_{0}$}, $\ord_{u}(z_{0})$, as follows: Take a neighborhood $N(z_{0}) \subset \Sigma$ containing $z_{0}$ within which $u$ admits a power-series expansion
\begin{equation*}
u|_{N(z_{0})} = \sum_{0}^{\infty} a_{k}z^{k}
\end{equation*}
where we have identified $z_{0} = 0\in N(z_{0})$. For $z_{0} \in \Int(\dot{\Sigma})$, the neighborhood hood $N(z_{0})$ may be modeled on $\B_{\epsilon}(0) \subset \C$ and for $z_{0} \in \partial \Sigma$, the neighborhood may be modeled on the half-ball $\B_{\epsilon}(0) \cap \{ |y| \geq 0 \}$ for some $\epsilon > 0$. Then $\ord_{u}(z_{0})$ is defined the smallest $k$ for which the $z^{k}$ coefficient in $du$ is non-zero.

\subsubsection{Orders of boundary punctures}

We will additionally need a notion of the order of $u$ at a boundary puncture. Here matters will be considerably simplified by our requirement that the angles of self-intersections of $\lambda$ are always $\frac{\pi}{2}$.

After a rotation, the local picture of $\lambda$ at a self-intersection will always look like the coordinate axes in $\C$ intersecting at the origin. We then define the order of $u$ at a boundary puncture $p_{i}$, $\ord_{u}(p_{i})$ to be the number of quadrants covered by the map $u$ by a neighborhood about $p_{i}$ minus $1$.

\begin{figure}[h]
\begin{overpic}[scale=.5]{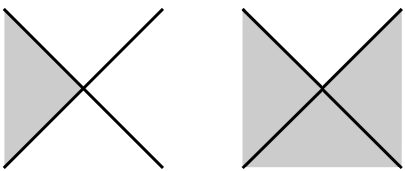}
\end{overpic}
\caption{Convex and non-convex corners corresponding to boundary punctures of orders $0$ and $2$, respectively. Each subfigure may be rotated by multiples of $\frac{\pi}{2}$.}
\label{Fig:ConvexNonconvex}
\end{figure}

For some examples, we'll work with the actual coordinate axes: Let $\lambda = \{ x = 0 \} \cup \{y = 0\}$ and take a neighborhood of $p_{i}$ in $\dot{\Sigma}$ to be the upper-right quadrant $\B \cap \{ x, y\geq 0\}$. Then if $u$ is locally given by $u(z) = z^{k}$ with $k$ odd, then $\ord_{u}(p_{i}) = k -1$. Observe that this number must be even. Figure \ref{Fig:ConvexNonconvex} shows boundary punctures of orders $0$ and $2$.

\subsection{Index calculations from branch orders}

Here we calculate the Maslov indices of maps $u \in \ModSpace_{u}$ from the orders of critical points and boundary punctures. 

\begin{thm}\label{Thm:BranchIndex}
For $u \in \ModSpace^{\lambda}$,
\begin{equation*}
\ind(u) = \half \sum_{p_{i} \in \partial \Sigma}\ord_{u}(p_{i})  + \sum_{z\in \Crit^{\partial}_{u}}\ord_{u}(z) + 2\sum_{z \in \Crit^{\Int}_{u}}\ord_{u}(z).
\end{equation*}
Hence all rigid curves in $\ModSpace^{\lambda}$ are immersed and such that every boundary puncture has order $0$.
\end{thm}

\begin{proof}
We will work through a sequence of increasingly general cases. In each case, we'll be computing Maslov numbers using rotation angles of paths as described in Equation \eqref{Eq:MaslovRotation}.
	
First consider the case when $\Sigma = \disk$, $\Crit(u) = \emptyset$, and there are no boundary punctures. Then $u$ is an immersion of the disk into $\C$, and the restriction of $u$ to $\partial \disk$ will have total rotation angle $2\pi$. Then $\Maslov(\partial \Sigma) = 2$ and $2\chi(\Sigma) = 2$ so that $\ind(u) = 0$, as expected.

Next, consider the case when $\Sigma = \disk$, $\Crit(u) = \emptyset$, and the order of $u$ at each boundary puncture $p_{i}$ is $0$. If we round each $\frac{\pi}{2}$ corner of the image of $u$, we obtain an immersion $\widetilde{u}$ whose domain is a disk without boundary punctures. The effect of rounding corners adds $\frac{\pi}{2}$ to the boundary rotation angle for each boundary puncture of $u$. Hence
\begin{equation*}
\sum \theta_{k} = 2\pi - \frac{\pi}{2}\#(p_{i}).
\end{equation*}
When computing the Maslov index of $u|_{\partial \Sigma}$, we subtract $\frac{\pi}{2}$ from the rotation angle for each boundary puncture. Hence $\Maslov(u|_{\partial \Sigma}) = 2 - \#(p_{i})$ and so $\ind(u) = 0$, establishing the theorem in this case.

\begin{figure}[h]
\begin{overpic}[scale=.4]{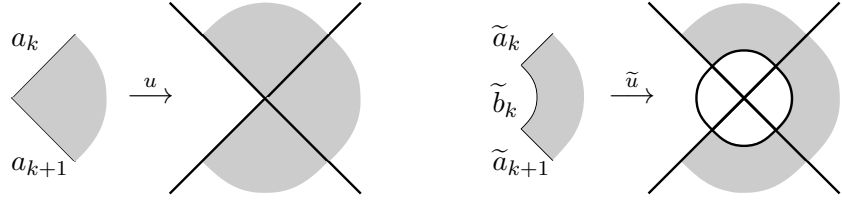}
\put(14, 11){$\overset{u}{\longrightarrow}$}
\put(72, 11){$\overset{\widetilde{u}}{\longrightarrow}$}
\put(0, 18){$a_{k}$}
\put(0, 3){$a_{k+1}$}
\put(58, 18){$\widetilde{a}_{k}$}
\put(58, 10){$\widetilde{b}_{k}$}
\put(58, 3){$\widetilde{a}_{k+1}$}
\end{overpic}
\caption{On the left, a neighborhood of a boundary puncture $p \in \partial \Sigma$ is sent via $u$ to a non-convex corner, so that $\ord_{u}(p) = 2$. On the right, we modifying $\lambda$ and $u$ to obtain $\widetilde{\lambda}$ and $\widetilde{u}$ having two boundary punctures, each having order $0$.}
\label{Fig:PizzaCorner}
\end{figure}

Now we consider the case in which $\Sigma = \disk$, $\Crit(u) = \emptyset$, and the orders of boundary punctures are allowed to be arbitrary non-negative integers. Let $\widetilde{\lambda}$ be an immersed multi-curve obtained from $\lambda$ by adding circles $\partial \disk_{\epsilon}(z_{k})$ centered about each of the double points $z_{k} \in \C$ of $\lambda$. A map $\widetilde{u}$ is obtained from $u$ by deleting the connected components $u^{-1}(\B_{\epsilon})$ from the domain of $u$ which contain punctures. See Figure \ref{Fig:PizzaCorner}. Then the domain of $\widetilde{u}$ will have $2\#(p_{i})$ boundary marked points and the order of $\widetilde{u}$ at each boundary marked point will be $0$. We have already established that if $\widetilde{\theta}_{k}$ are the rotation angles over the boundary components of the domain $\widetilde{\Sigma}$ of $\widetilde{u}$, then 
\begin{equation}\label{Eq:TildeThetaAngles}
\sum \widetilde{\theta}_{k} = 2\pi - \pi\#(p_{i})
\end{equation}
We label the boundary components of $\partial \dot{\Sigma}$ and $\partial \widetilde{\Sigma}$ as shown in Figure \ref{Fig:PizzaCorner} so that each connected component $a_{k}$ of $\partial \dot{\Sigma}$ determines a connected component $\widetilde{a}_{k}$ of $\partial \widetilde{\Sigma}$ and each $p_{k}$ determines a new boundary arc $\widetilde{b}_{k}$ of $\partial \widetilde{\Sigma}$. From the figure, we see that the rotation angle of each $\widetilde{a}_{k}$ agrees with the rotation angle of each $a_{k}$. The new boundary arc, considered as living in $\dot{\Sigma}$ is exactly as the form of the arcs used to define $\ord_{u}(p)$ for a boundary puncture except oriented backwards. The rotation angle of $\tilde{b}_{k}$ is $- \frac{\pi}{2}(\ord_{u}(p_{k}) + 1)$. By combining the rotation angles $\theta(\widetilde{b}_{k})$ with Equation \eqref{Eq:TildeThetaAngles}, we conclude
\begin{equation}\label{Eq:AngleSumPunctures}
\begin{aligned}
\sum \theta(a_{k}) &= \sum \theta\left(\widetilde{a}_{k}\right) = \sum \theta\left(\widetilde{a}_{k}\right) + \sum \theta\left(\widetilde{b}_{k}\right) - \sum \theta\left(\widetilde{b}_{k}\right) \\
& = \sum \widetilde{\theta}_{k} + \frac{\pi}{2}\sum \left(\ord_{u}(p_{k}) + 1\right) \\
&= 2\pi - \frac{\pi}{2} \#(p_{k}) + \frac{\pi}{2} \sum \ord_{u}(p_{k}).
\end{aligned}
\end{equation}
Then $\Maslov(u|_{\partial \Sigma}) = 2 - \#(p_{i}) + \half \sum \ord_{u}(p_{i})$ so that
\begin{equation*}
\ind(u) = 2 - \#(p_{i}) + \half\sum \ord_{u}(p_{i}) - 2\chi(\Sigma) + \#(p_{i}) = \half\sum \ord_{u}(p_{i}).
\end{equation*}
Thus we have established the theorem in our current case.

\begin{figure}[h]
\begin{overpic}[scale=.4]{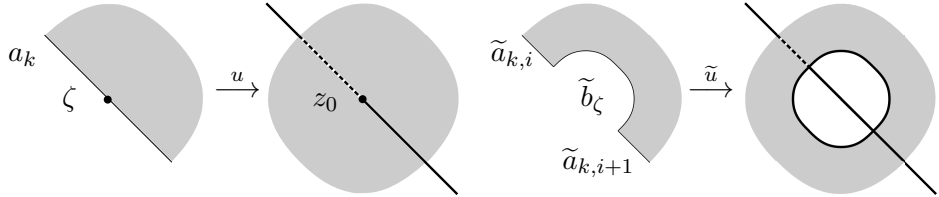}
\put(19, 11){$\overset{u}{\longrightarrow}$}
\put(72, 11){$\overset{\widetilde{u}}{\longrightarrow}$}
\put(-4, 15){$a_{k}$}
\put(2, 10){$\zeta$}
\put(30, 10){$z_{0}$}
\put(50, 15){$\widetilde{a}_{k, i}$}
\put(60, 10){$\widetilde{b}_{\zeta}$}
\put(58, 3){$\widetilde{a}_{k, i+1}$}
\end{overpic}
\caption{Modifying $\lambda$ and $u$ to obtain $\widetilde{\lambda}$ and $\widetilde{u}$ near a boundary critical point $\zeta$ of $\ord_{u}(\zeta) = 1$ with $u(\zeta) = z_{0}$. The map $u$ is obtained locally by folding the boundary in half at the point $\zeta$ with $\partial \Sigma$ mapped to the dashed arc.}
\label{Fig:PizzaBoundary}
\end{figure}

Now we add boundary critical points to our calculation, applying a similar trick wherein $\lambda$ and $u$ are modified. See Figure \ref{Fig:PizzaBoundary}. Suppose that $\Sigma$ is a disk, $\Crit(u)  = \Crit^{\partial}(u)$, and the orders of the boundary punctures are arbitrary non-negative integers. Write $z_{k} \in \C$ for the critical values of $u$. We add circles $\partial \disk_{\epsilon}(z_{k})$ about each $z_{k}$ to $\lambda$ to obtain an immersed multi-curve $\widetilde{\lambda}$. A half-ball is cut out from $\dot{\Sigma}$ centered about each point $\zeta \in \Crit(u)$ to obtain a domain $\widetilde{\Sigma} \subset \Sigma$, a disk with corners, and a holomorphic map $\widetilde{u}: \widetilde{\Sigma} \rightarrow \C$ with boundary on $\widetilde{\lambda}$. The newly created boundary components, $\widetilde{b}_{\zeta}$, will be mapped to the circles which have been added to $\lambda$. For each connected component $a_{k}$ of $\partial \dot{\Sigma}$, there will be a collection of boundary components $\widetilde{a}_{k, i}$. The endpoints of the $\widetilde{b}_{\zeta}$ provide two new marked points on $\partial \widetilde{\Sigma}$ which will both be asymptotic to a point of intersection of $\lambda$ with one of the circles we've added, lying in the $u(\partial \dot{\Sigma}) \subset \lambda$. The map $\widetilde{u}$ is as described in the previous case considered, without boundary critical points and with 
\begin{equation*}
\#(\widetilde{p}_{i}) = \#(p_{i}) + 2\#\Crit^{\partial}(u), \quad \sum_{i} \theta(\widetilde{a}_{k, i}) = \theta(a_{k}),\quad \theta(\widetilde{b}_{\zeta}) = -\pi(\ord_{u}(\zeta) + 1).
\end{equation*}
Here $\tilde{p}_{i}$ are the boundary punctures of $\tilde{\Sigma}$. Those $\tilde{p}_{i}$ corresponding to the $p_{i}$ of $\Sigma$ have $\ord_{\tilde{u}}(\tilde{p}_{i}) = \ord_{u}(p_{i})$ and the order of $\widetilde{u}$ at each of the two new $\widetilde{p}_{i}$ is $0$. So we calculate
\begin{equation}\label{Eq:AngleSumBoundary}
\begin{aligned}
\sum \theta(a_{k}) &=  \sum_{k, i} \theta\left(\widetilde{a}_{k, i}\right) \\
&=  \left(\sum_{k, i} \theta\left(\widetilde{a}_{k, i}\right) + \sum_{\zeta \in \Crit(u)} \theta(\widetilde{b}_{\zeta})\right) - \sum_{\zeta \in \Crit(u)} \theta\left(\widetilde{b}_{\zeta}\right) \\
&= \left( 2\pi - \frac{\pi}{2}\#\left(\widetilde{p}_{k}\right) + \pi \sum \ord_{u}\left(\widetilde{p}_{k}\right) \right) + \pi \sum_{\zeta \in \Crit(u)} \left(\ord_{u}(\zeta) + 1\right) \\
&= 2\pi - \frac{\pi}{2}\left(\#(p_{k}) + 2\#\Crit^{\partial}(u)\right) + \frac{\pi}{2} \sum \ord_{u}(p_{k}) + \pi \sum_{\zeta \in \Crit(u)} \left(\ord_{u}(\zeta) + 1\right)\\
&= 2\pi - \frac{\pi}{2}\#(p_{k}) + \frac{\pi}{2} \sum \ord_{u}\left(p_{k}\right) + \pi \sum_{\zeta \in \Crit^{\partial}_{u}(\zeta)} \ord_{u}(\zeta).
\end{aligned}
\end{equation}
Here we are applying Equation \eqref{Eq:AngleSumPunctures} in the third line above. Again, we use the rotation angle computation above to compute the Maslov number of $u|_{\partial \Sigma}$, establishing the theorem in the case under consideration.

\begin{figure}[h]
\begin{overpic}[scale=.4]{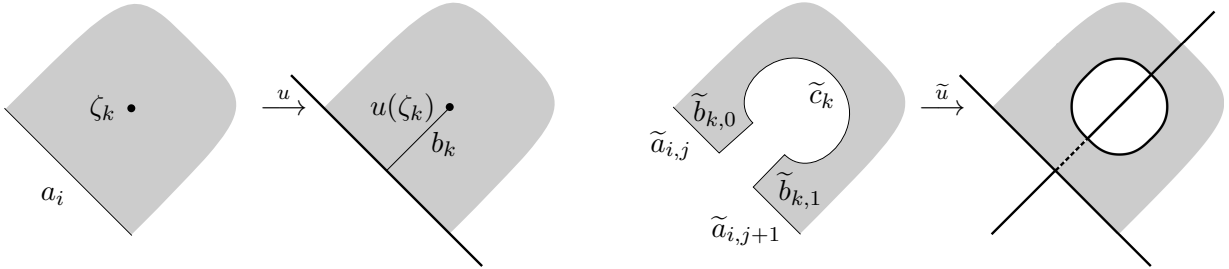}
\put(21, 12.5){$\overset{u}{\longrightarrow}$}
\put(75, 12.5){$\overset{\widetilde{u}}{\longrightarrow}$}
\put(3, 5.5){$a_{i}$}
\put(7, 12.5){$\zeta_{k}$}
\put(30, 12.5){$u(\zeta_{k})$}
\put(35, 9.5){$b_{k}$}
\put(53, 9.5){$\widetilde{a}_{i, j}$}
\put(56.5, 12){$\widetilde{b}_{k, 0}$}
\put(66, 13.5){$\widetilde{c}_{k}$}
\put(63.5, 5.5){$\widetilde{b}_{k, 1}$}
\put(58, 2.5){$\widetilde{a}_{i, j+1}$}
\end{overpic}
\caption{Modifying $\lambda$ and $u$ to account for interior critical points. The dashed arc is the image of $\widetilde{b}_{k, j}$ which overlap perfectly.}
\label{Fig:PizzaInterior}
\end{figure}

Next we incorporate interior critical points into our calculation, following a modification described in Figure \ref{Fig:PizzaInterior}. We start with a general holomorphic map $u$ whose domain is a disk with boundary punctures $\dot{\Sigma}$, mapping $\partial \dot{\Sigma}$ to some $\lambda$. We write $\zeta_{i}$ for points in $\Crit_{u}^{\Int}$ and $z_{k} \in \C$ for the critical values of $u$. For each $z_{k}$, let $\gamma$ be an embedded arc in $\C$ satisfying the following conditions:
\be
\item The $\gamma_{k}$ are pairwise disjoint as subsets of $\C$.
\item For each $k$, $z_{k}$ is a unique point in $\gamma_{k} \cap u(\Crit(u))$.
\item Each $\gamma_{k}$ intersections $\lambda$ at least once.
\item Each $\gamma_{k}$ intersects $\lambda$ transversely in $\frac{\pi}{2}$ angles.
\ee
Let $\widetilde{\lambda}$ be the union of $\lambda$ with the $\gamma_{k}$ and circles $\partial \disk_{\epsilon}(z_{k})$ about each $z_{k}$. For each $\zeta_{k}$ choose a path $b_{k} \subset \gamma_{k} \cap \im(u)$ which starts on $\lambda$ and ends at the point $z_{k}$. We require that all of the $b_{k}$ have disjoint interiors. For each $b_{k}$, choose a path $\widetilde{b}_{k} \subset \dot{\Sigma}$ which starts on $\partial \dot{\Sigma}$ and ends at $\zeta_{k}$. To obtain our $\widetilde{\dot{\Sigma}}$, we cut a keyhole-type shape out of $\dot{\Sigma}$ at each $\zeta_{k}$ as follows:
\be
\item Remove the connected component of $u^{-1}(\B_{\epsilon}(z_{k}))$ which contains the point $\zeta_{k}$.
\item Remove the arc $\widetilde{b}_{k}$.
\item Compactify along what was $\widetilde{b}_{k}$ to obtain boundary arcs $\widetilde{b}_{k, 0}, \widetilde{b}_{k, 1}$.
\ee
The result of this surgery, shown in the center-right of Figure \ref{Fig:PizzaInterior}, is a surface $\widetilde{\dot{\Sigma}}$ with four new boundary marked points for each $\zeta_{k}$. We retroactively specify that the $\widetilde{b}_{k}$ are chosen so that the surface $\widetilde{\Sigma}$ is connected. The map $\widetilde{u}$ is obtained by restricting $u$ to the interior of $\widetilde{\Sigma}$ which is contained in the interior of $\dot{\Sigma}$. The newly created boundary punctures will be asymptotic to intersections of $\gamma_{k}$ with $\lambda$ and of $\gamma_{k}$ with $\partial \disk_{\epsilon}(z_{k})$. Clearly the order of $\widetilde{u}$ at each of the new punctures is $0$. In addition to the $\widetilde{b}_{k, j}$, the boundary components $a_{i}$ will be subdivided into $\widetilde{a}_{i, j}$, and there is a new boundary arc for each $\zeta_{k}$, which we will denote $\widetilde{c}_{k}$ as shown in Figure \ref{Fig:PizzaInterior}. For the calculations of rotation angles, we have
\begin{equation*}
\sum \theta(a_{i}) = \sum_{j} \theta\left(\widetilde{a}_{i j}\right),\quad \theta\left(\widetilde{b}_{k, 0}\right) = - \theta\left(\widetilde{b}_{k, 1}\right),\quad \theta\left(\widetilde{c}_{k}\right) = -2\pi \left(\ord_{u}(\zeta_{k}) + 1\right).
\end{equation*}
As the map $\widetilde{u}$ satisfies the conditions of Equation \eqref{Eq:AngleSumBoundary}, we apply that formula together with the above expressions and $\#(\widetilde{p}_{k}) = \#(p_{k}) + 4\#\Crit^{\Int}(u)$ to obtain
\begin{equation*}
\begin{aligned}
\sum \theta(a_{k}) &=  \sum_{i} \theta\left(\widetilde{a}_{k, i}\right) \\
&=  \left( \sum_{i} \theta\left(\widetilde{a}_{k, i}\right) + \sum_{\zeta \in \Crit^{\Int}(u)} \theta\left(\widetilde{c}_{\zeta}\right) \right) - \sum_{\zeta \in \Crit^{\Int}(u)} \theta\left(\widetilde{c}_{\zeta}\right) \\
&= \left( 2\pi - \frac{\pi}{2}\#\left(\widetilde{p}_{k}\right) + \frac{\pi}{2} \sum \ord_{u}\left(\widetilde{p}_{k}\right) + \pi \sum_{\zeta \in \Crit^{\partial}_{u}(\zeta)} \ord_{u}(\zeta)\right) + 2\pi \sum_{\zeta \in \Crit^{\Int}(u)} \left(\ord_{u}(\zeta) + 1\right) \\
&= 2\pi - \frac{\pi}{2}\#(p_{k}) + \frac{\pi}{2} \sum \ord_{u}(p_{k}) + \pi \sum_{\zeta \in \Crit^{\partial}_{u}(\zeta)} \ord_{u}(\zeta) + 2\pi \sum_{\zeta \in \Crit^{\Int}(u)} \ord_{u}(\zeta)
\end{aligned}
\end{equation*}
proving the theorem in full generality for $\Sigma = \disk$.

\begin{figure}[h]
\begin{overpic}[scale=.5]{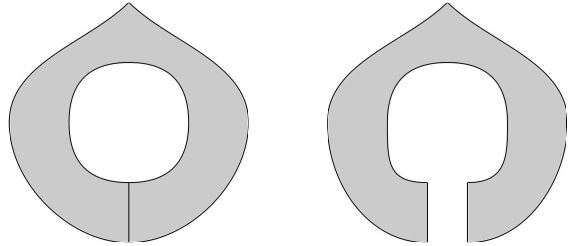}
\end{overpic}
\caption{Cutting up $\dot{\Sigma}$ using an arc basis to obtain $\widetilde{\dot{\Sigma}} \subset \disk$. The case of an annulus is shown, wherein there is only a single $b_{k}$.}
\label{Fig:PizzaArcBasis}
\end{figure}

To complete our calculation we must consider maps $u$ whose domain may have $\chi(\Sigma) < 1$. In this, let $b_{k}$ be an arc basis of $\dot{\Sigma}$ for which each $b_{k}$ is disjoint from $\Crit(u)$ and the boundary marked points $p_{k}$. By compactifying $\dot{\Sigma} \setminus (\cup b_{k})$ as shown in Figure \ref{Fig:PizzaArcBasis} we obtain a disk with $\#(p_{k}) + 4\dim(H_{1}(\Sigma))$ boundary punctures, whose complement we denote by $\widetilde{\Sigma}$. We also get a map $\widetilde{u}$, which we may consider as having boundary on an immersed multi-curve $\widetilde{\lambda}$ obtained by taking the union of $\lambda$ with simple closed curves containing the $u(b_{k})$. Following the above calculations of $\sum \theta(a_{k})$, we subtract $2\pi$ for each $b_{k}$. Thus,
\begin{equation*}
\frac{1}{\pi} \sum \theta(a_{k}) =  2\chi(\Sigma) - \half\#(p_{k}) + \half\sum \ord_{u}(p_{k}) + \sum_{\zeta \in \Crit^{\partial}(u)} \ord_{u}(\zeta) + 2\sum_{\zeta \in \Crit^{\Int}(u)} \ord_{u}(\zeta).
\end{equation*}
By combining this formula with Theorem \ref{Thm:IndMaslovTwoD}, the proof is complete.
\end{proof}

\section{Examples of holomorphic annuli}\label{Sec:AnnuliExamples}

In this section we study examples of holomorphic annuli on Legendrian links using our index calculations and analysis of cohomological defects. Our intentions are to
\be
\item demonstrate the utility of these tools, showing how the obstruction class can be used to count curves in certain restricted scenarios,
\item point out some obstacles in reducing counts of $\chi < 1$ holomorphic curves to combinatorics, and
\item provide contrast between moduli spaces of holomorphic curves associated to arbitrary Legendrian links, those which are right-simple, and those which are left-right-simple.
\ee

\subsection{An index $0$ annulus in $\ModSpace^{\Lambda}$}

We explicitly describe an $\ind(U) = 0$ annulus in $\ModSpace^{\Lambda}$. Although the existence of such annuli violate an expected dimension count, this example will help to us to understand some more subtle counting difficulties in Section \ref{Sec:ImpossibleCount} below.

\begin{figure}[h]
	\begin{overpic}[scale=.4]{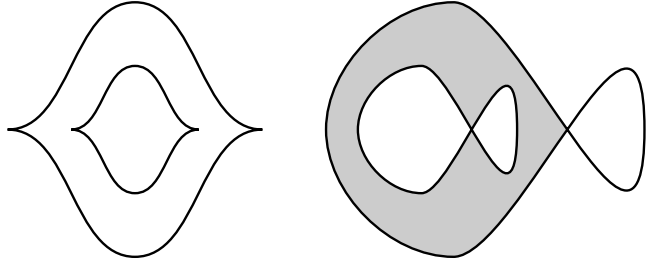}
	\end{overpic}
	\caption{A holomorphic annulus on the Lagrangian resolution of a pair of unknots.}
	\label{Fig:ConcentricUnknot}
\end{figure}

Consider the $2$-component unlink of Legendrian $\tb = -1$ unknots shown in Figure \ref{Fig:ConcentricUnknot}. Let's call this link $\Lambda_{0}$ and for $T \in \R$ write $\Lambda_{T}$ for the link obtained by translating the inner knot in the $t$-direction of $\R^{3}$ by $T\in \R$, while keeping the outer knot fixed. All of the $\Lambda_{T}$ have the same Lagrangian projections which will be denoted by $\lambda$. As shown in the right-hand side of the figure, there is a rigid holomorphic annulus on $\lambda$, one of whose boundary components is positively asymptotic to a self intersection of the outer component of $\lambda$. The other boundary component has two punctures negatively asymptotic to the single self-intersection of the inner component of $\lambda$.

Suppose that the annuli $\Sigma_{T}$ in the domain of these maps are isomorphic to $[-C, C]_{p} \times \Circle_{q}$ for some $C > 0$ with $\Circle = [0, 2\pi]/\sim$ and the $p = -C$ boundary component sent to the inner component of $\lambda$. Note that $C$ is constant in $T$ as the maps $u$ are unchanged by the variation in $\Lambda_{T}$. Let $t_{T} \in \Cinfty(\dot{\Sigma}_{T})$ be harmonic functions extending the Dirichlet boundary condition $t^{\partial}_{T} \in \Cinfty(\partial \dot{\Sigma}_{T})$ associated to each $\Lambda_{T}$. We apply a Fourier series expansion to the restriction of $t_{T}$ to each $\{ p = p_{0}\}$ to obtain the expression
\begin{equation}\label{Eq:tTof}
t_{T}(p, q) = \sum_{n \in \Z} f_{T, n}(p)e^{inq},\quad f_{T, n} \in \Cinfty([-C, C], \C), \quad f_{T, -n}(p) = \overline{f_{T, n}}(p).
\end{equation}
We compute the Laplacian of $t_{T}$ as
\begin{equation*}
\quad \Delta t_{T}(p, q) = \Big( \sum_{n \in \Z} \big( \frac{\partial^{2} f_{T, n}}{\partial p^{2}}(p) -n^{2}f_{T, n}(p) \big)e^{inq} \Big)dp\wedge dq = 0,
\end{equation*}
telling us that for all $T$, $f_{T, 0}$ must satisfy $\frac{\partial^{2} f_{T, 0}}{\partial p^{2}}(p) = 0$, and so be linear functions of $p$. We may then write $f_{T, 0}(p) = a_{T}p + b_{T}$ for some real-values constants (functions of the variable $T$), $a_{T}$ and $b_{T}$. Using this expression, we compute the integral of our obstruction class $\obstruction_{T} \in H^{1}(\Sigma_{T})$ as
\begin{equation*}
\begin{aligned}
\int_{\{p = 0\}} dt_{T}\circ j &= \int_{\{p = 0\}} \sum_{n \in \Z} e^{inq}\Big( \frac{\partial f_{T, n}}{\partial p}dp + in f_{T, n}dq\Big)\circ j \\
&= \int_{\{p = 0\}} \sum_{n \in \Z} e^{inq}\Big( -\frac{\partial f_{T, n}}{\partial p}dq + in f_{T, n}dp\Big) \\
&= \int_{\{p = 0\}} -\frac{\partial f_{T, 0}}{\partial p}dq = -2\pi a_{T} \\
&= \frac{1}{2C}\int_{\Sigma_{T}}-\frac{\partial f_{T, 0}}{\partial p}dp \wedge dq \\
&=\frac{1}{2C}\Big( \int_{p = -C}f_{T, 0}dq - \int_{p = C}f_{T, 0}dq  \Big) = \frac{1}{2C}\Big( \int_{p = -C}t^{\partial}_{T}dq - \int_{p = C}t^{\partial}_{T}dq  \Big).
\end{aligned}
\end{equation*}
At the third and fourth lines we use the fact that $f_{T, 0}$ is an affine function of $p$. In the last line, we use the fact that $\int_{p = \pm C} f_{T, n}e^{inq}dq = 0$ for each $n \neq 0$ together with the expression for $t_{T}$ in Equation \eqref{Eq:tTof}.

For $T > 0$ large enough that $\inf_{\{ p = -C\}}t^{\partial}_{T} > \sup_{\{ p = C \}}t^{\partial}_{T}$ the above expression will be positive, while for $T < 0$ small enough that $\sup_{\{ p = -C\}}t^{\partial}_{T} < \inf_{\{ p = C \}}t^{\partial}_{T}$ the above expression will be negative. Hence for some $T \in \R$, the obstruction class vanishes by the intermediate value theorem, and the annulus admits a holomorphic lift to the symplectization.

Using the above methods, we only have control over $\obstruction$ for $|T|$ large. We therefore cannot determine its behavior from only looking at the Lagrangian resolution of the front projection. Also problematic is the fact that (unbranched) multiple covers of a $\obstruction = 0$ annulus will also have $\ind(U) = 0$.

\subsection{A rigid holomorphic annulus in $\ModSpace^{\Lambda}$}\label{Sec:RigidUAnnulus}

To demonstrate that $\ind(U)=1$ holomorphic annuli can exist when using the simple fronts of \cite{Ng:ComputableInvariants}, we look to the link $\Lambda$ appearing in Figure \ref{Fig:RSimpleAnnulus}. On the right-hand side of the figure, we see a holomorphic annulus of index $1$ on the Lagrangian submanifold $\lambda = \pi_{\C}(\Lambda) \subset \C$. This annulus lives in a $1$-dimensional moduli space of such annuli, $\ModSpace^{\lambda}_{u}$ which we will shortly describe. Working under the assumption that $\obstruction$ is transverse to $0$, we will see that the signed count $\#(\obstruction^{-1}(0))$ of elements of $\ModSpace^{\lambda}$ which lift to holomorphic annuli in $\ModSpace^{\Lambda}/\R$ has absolute value $1$.

\begin{figure}[h]
	\begin{overpic}[scale=.4]{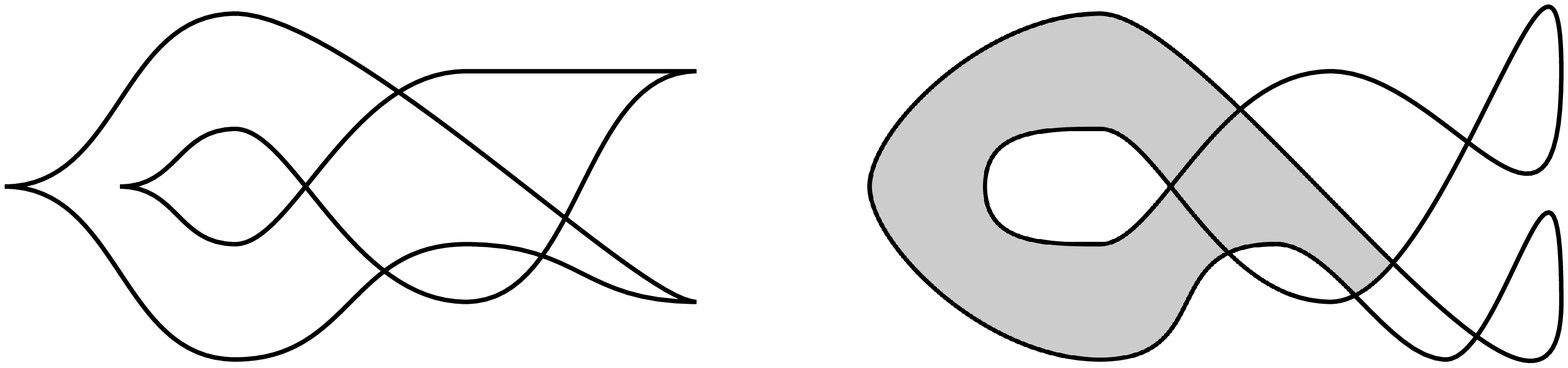}
		\put(78, 19){$\chord_{1}$}
		\put(70, 11){$\chord_{2}$}
		\put(78, 3){$\chord_{3}$}
	\end{overpic}
	\caption{An index $1$ annulus on the Lagrangian projection of a right-simple Legendrian link.}
	\label{Fig:RSimpleAnnulus}
\end{figure}

In Figure \ref{Fig:RSimpleAnnulusFamily}, we consider the moduli space $\ModSpace^{\lambda}_{u}$ which is parameterized by the image of a single critical point appearing on the boundary of the annulus. As in previous figures, we use dashed arcs to indicate where the image of the boundary of a curve double-covers $\lambda$. From our previous calculations, we expect $\dim(\ModSpace^{\lambda}_{u}) = 1$ and since $\dim H^{1}(\Sigma)$, we expect that $\bigO^{-1}(0)$ -- described in Theorem \ref{Thm:ObstructionCount} will be a collection of points.

\begin{figure}[h]
	\begin{overpic}[scale=.35]{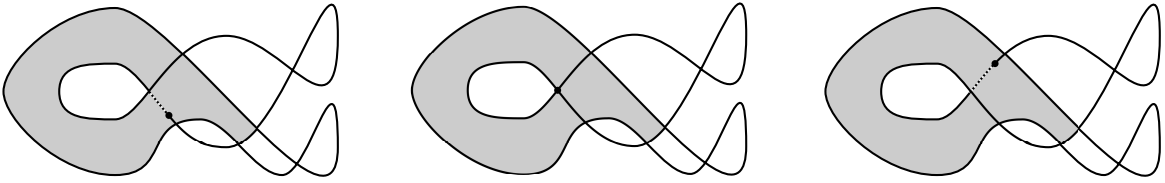}
	\end{overpic}
	\caption{The moduli space $\ModSpace_{u}$ is parameterized by the location of a single boundary branch point, shown as a thick dot.}
	\label{Fig:RSimpleAnnulusFamily}
\end{figure}

We seek to count the number of points in $\bigO^{-1}(0)$ algebraically as a kind of relative Euler number over $\ModSpace^{\lambda}_{u}$. Let $\gamma$ be the oriented loop in $\Sigma$ shown in the center of Figure \ref{Fig:RSimpleAnnulusNodal}. Identify $H^{1}(\Sigma)$ with $\R$ via the mapping $[\beta] \mapsto \int_{\gamma}\beta$ so that $\bigO$ may be viewed as a real number. 

\begin{figure}[h]
	\begin{overpic}[scale=.45]{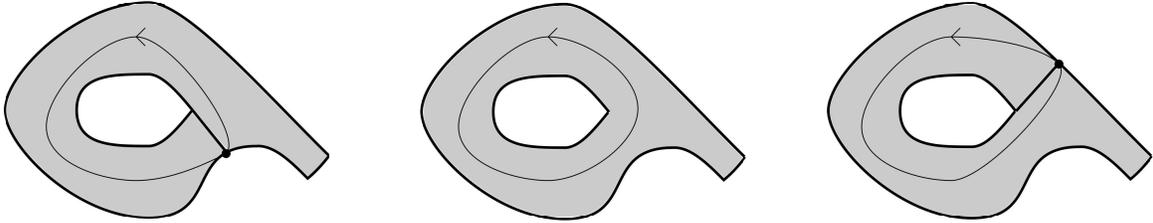}
	\end{overpic}
	\caption{Nodal degeneration of $\Sigma$.}
	\label{Fig:RSimpleAnnulusNodal}
\end{figure}

At one end of our moduli space shown in the left subfigure of Figure \ref{Fig:RSimpleAnnulusNodal}, $\Sigma$ degenerates into a nodal curve which can be viewed as a disk with two boundary punctures identified. A holomorphic map $u_{0}$ from this disk (with the node forgotten) will lift to the symplectization $\R \times \R^{3}$ with boundary on $\R \times \Lambda$ via a map $U_{0}$. Let $\gamma_{0}$ be the arc shown in the left-hand side of the figure. The integral of $dt \circ j$ of this arc may be computed
\begin{equation*}
\int_{\gamma_{0}} dt \circ j = \int_{\gamma_{0}} ds = \infty
\end{equation*}
as the arc begins at a negative puncture and ends at a positive puncture. We conclude that $\int_{[\gamma]} dt\circ j$ tends to infinity as the maps $u$ tend to the nodal limiting curve on the left-hand side of Figure \ref{Fig:RSimpleAnnulusNodal}.

A similar analysis may be applied to the right-hand side of the figure. The arc shown starts at a positive puncture of a map $U_{1}: \dot{\disk} \rightarrow \R \times \R^{3}$ and ends at negative puncture. We conclude that $\int_{[\gamma]} dt\circ j \rightarrow -\infty$ at this end of the moduli space. As $\bigO \rightarrow \infty$ at one end of our moduli space, and $\bigO \rightarrow -\infty$ at the end other end we conclude that the algebraic count of points in $\bigO^{-1}(0)$ is $\pm 1$. Hence at least $1$ point in $\ModSpace^{\lambda}_{u}$ must determine a rigid annulus in the associated moduli space $\ModSpace_{U}^{\Lambda}$.

\subsection{An impossible-to-count annulus}\label{Sec:ImpossibleCount}

We combine concepts of the preceding two examples to provide an example of a $1$-dimensional family of $u$ annuli for which $\obstruction^{-1}(0)$ is impossible to count from the Lagrangian projection.

\begin{figure}[h]
\begin{overpic}[scale=.5]{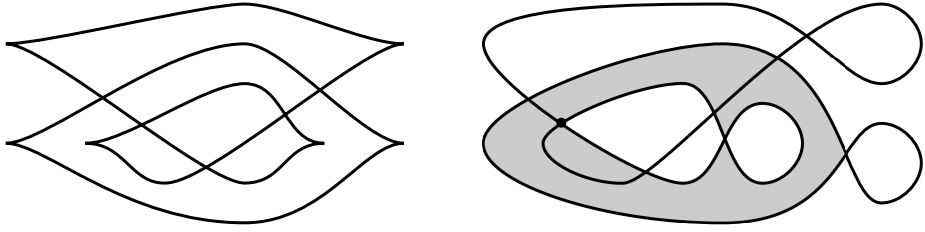}
\end{overpic}
\caption{A Legendrian link in the front and Lagrangian projections, together with the image of an $\ind(u) = 1$ annulus in $\C$.}
\label{Fig:BadAnnulusLink}
\end{figure}

Figure \ref{Fig:BadAnnulusLink} shows a Legendrian Hopf link in the front projection, together with its Lagrangian resolution. In the Lagrangian projection the image of a holomorphic map $u$ whose domain is an annulus is shown. The curve has $\ind(u) = 1$ and the associated moduli space $\ModSpace^{\lambda}_{u}$ may be parameterized by the location of a boundary branch point or -- at one point in the moduli space -- a non-convex corner, indicated by a dot in the figure. Let's say that $\ModSpace^{\lambda}_{u}$ is parameterized with a variable $T \in (0, 1)$.

\begin{figure}[h]
\begin{overpic}[scale=.5]{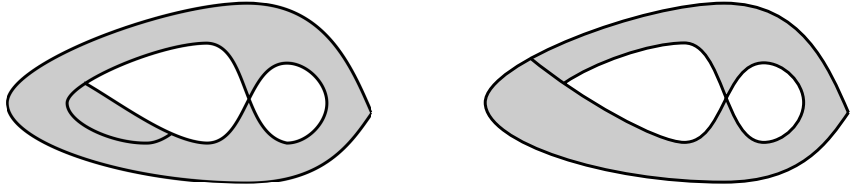}
\end{overpic}
\caption{$T \rightarrow 0$ and $T \rightarrow 1$ ends of the moduli space $\ModSpace^{\lambda}_{u}$.}
\label{Fig:BadAnnulusLinkCurves}
\end{figure}

As in the example of Section \ref{Sec:RigidUAnnulus}, we fix a homotopy class of simple closed curve $\gamma$ in the annulus and study $\int_{\gamma} \obstruction$ as the location of the branch point varies. At the end of the moduli space shown in the left-hand side of Figure \ref{Fig:BadAnnulusLinkCurves}, the annulus breaks into a rigid annulus and a disk. Let's say this is the $T \rightarrow 0$ end of $\ModSpace^{\lambda}_{u}$. As $\int_{\gamma} \obstruction$ is well defined over this new rigid annulus, we conclude that the limit $\lim_{T \rightarrow 0}\int_{\gamma} \obstruction(T)$ exists, yielding some $C \in \R$. Generically we may assume that $C \neq 0$.

The $T \rightarrow 1$ end of the moduli space is shown on the right-hand side of Figure \ref{Fig:BadAnnulusLinkCurves}. If, as in Figure \ref{Fig:RSimpleAnnulusNodal}, our curve $\gamma$ traverses the annulus counter-clockwise, then $\lim_{T \rightarrow 1}\int_{\gamma} \obstruction(T) = \infty$. Therefore we obtain
\begin{equation*}
|\#(\ModSpace_{U}^{\Lambda})| = \begin{cases}
1 & C < 0 \\
0 & C > 0.
\end{cases}
\end{equation*}
Unfortunately, it is impossible to read the sign of $C$ -- and so determine $|\#(\ModSpace_{U}^{\Lambda})|$ -- from only looking at the Lagrangian projection.

\subsection{Some non-rigid holomorphic annuli}\label{Sec:NonRigidUAnnulus}

Here we study a $1$-dimension family of annuli that reveal themselves when gluing together two Legendrian $RSFT$ disks at multiple punctures simultaneously. The Lagrangian projection $\lambda = \pi_{\C}(\Lambda)$ of a right-handed $\tb=1$ trefoil $\Lambda$ is depicted in Figure \ref{Fig:TrefoilDisks}. We again assume that $\obstruction$ is transverse to $0$.

\begin{figure}[h]
	\begin{overpic}[scale=.5]{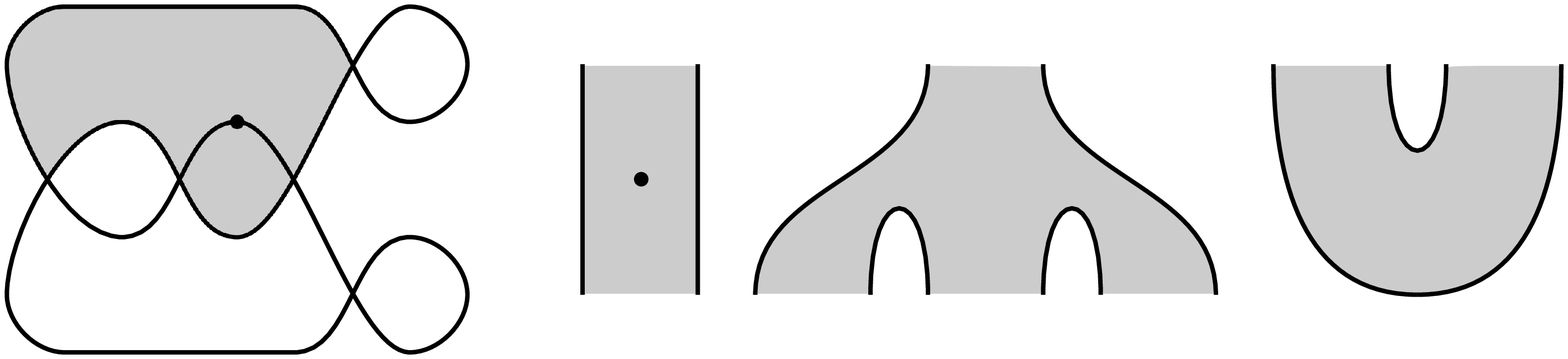}
		\put(-2, 10){$\chord_{1}$}
		\put(7, 10){$\chord_{2}$}
		\put(20, 10){$\chord_{3}$}
		\put(25, 18){$\chord_{4}$}
		\put(25, 3){$\chord_{5}$}
		\put(10, 17){$\region_{1}$}
		\put(13, 10){$\region_{2}$}
		\put(61, 10){$\region_{1}$}
		\put(89, 10){$\region_{2}$}
		
		\put(40, 20){$\chord_{4}$}
		\put(40, 0){$\chord_{1}$}
		
		\put(62, 20){$\chord_{4}$}
		\put(50, 0){$\chord_{1}$}
		\put(62, 0){$\chord_{2}$}
		\put(74, 0){$\chord_{3}$}
		
		\put(84, 20){$\chord_{2}$}
		\put(95, 20){$\chord_{3}$}
	\end{overpic}
	\caption{Holomorphic disks in $\R \times \R^{3}$ with boundary on $\R \times \Lambda$, projected to the $xy$-plane.}
	\label{Fig:TrefoilDisks}
\end{figure}

The trefoil has five chords labeled $\chord_{1},\dots, \chord_{5}$. There are two connected components of $\C \setminus \pi_{\C}(\Lambda)$ of interest to us, labeled $\region_{1}$ and $\region_{2}$. We can view the $\region_{i}$ as the interiors of images of holomorphic maps $u_{i}$ from disks with boundary punctures removed to the $xy$-plane. The union of the $\region_{i}$ also forms the image interior of the image of a holomorphic map, $u_{\bullet}$, whose domain is a disk.

The holomorphic maps $u_{\bullet}, u_{1}, u_{2}$ can be lifted to holomorphic maps 
\begin{equation*}
U_{\ast} = (s_{\ast}, t_{\ast}, u_{\ast}): \dot{\disk}_{\ast} \rightarrow \R \times \R \times \C,\quad U_{\ast}(\partial \dot{\disk}_{\ast}) \subset \R \times \Lambda
\end{equation*}
for $\ast = \bullet, 1, 2$. Here $\dot{\disk}_{\ast}$ is the unit disk in $\C$ with some collection of boundary punctures removed. Such lifts are unique up to translation in the $s$ coordinate and each disk has index $1$. The asymptotics for boundary punctures are indicated in the right half of Figure \ref{Fig:TrefoilDisks}. The map $U_{\bullet}$ is such that there a point in the interior of $\disk$ which intersects $\R \times \Lambda$, which is labeled with a dot. We assume that such an intersection is unique at which point $U (\disk)$ and $\R \times \Lambda$ meet transversely.

\begin{figure}[h]
	\begin{overpic}[scale=.5]{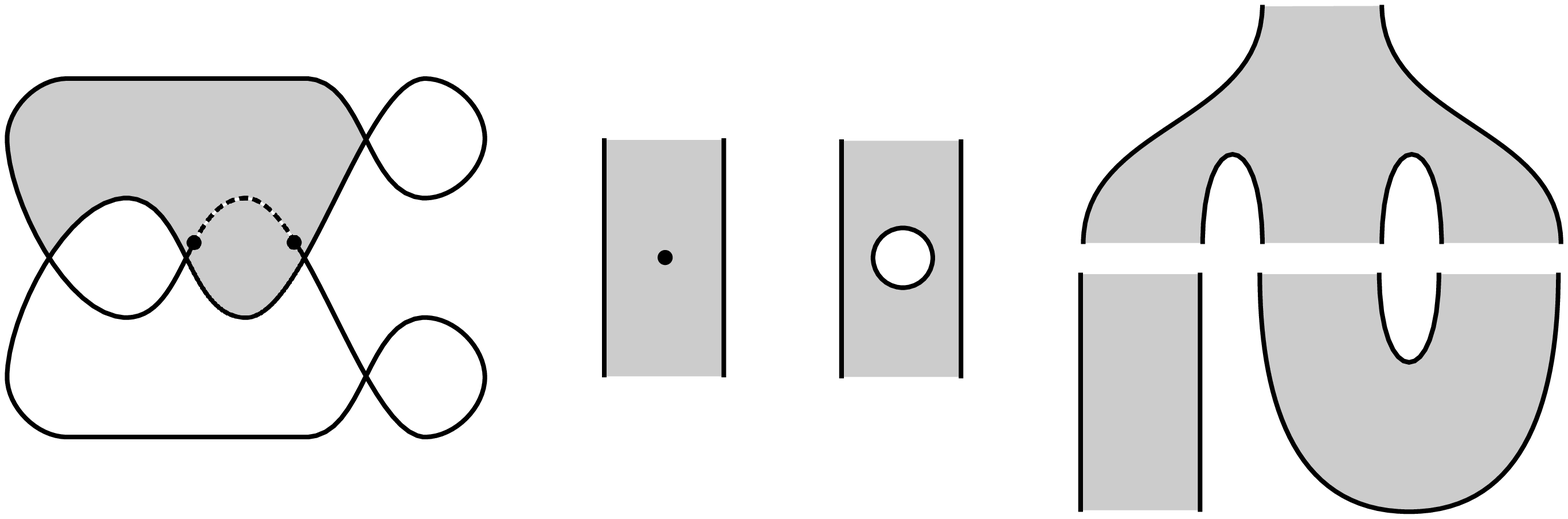}
	\end{overpic}
	\caption{The holomorphic maps $U_{T}$.}
	\label{Fig:TrefoilDiskGluings}
\end{figure}

If we glue the negative punctures of $U_{1}$ to the positive punctures of $U_{2}$ the positive punctures we expect to see a $2$-dimensional family of holomorphic maps from the annulus with two boundary punctures into $\R \times \R \times \C$. Modding out by translations in the $s$ direction, this yields a family of maps $U_{T}$ with 
\be
\item $\lim_{T\rightarrow \infty}U_{t}$ yielding the height $2$ holomorphic building determined by the gluing shown on the right-hand side of Figure \ref{Fig:TrefoilDiskGluings}.
\item $\lim_{T \rightarrow -\infty}U_{t}$ being the curve corresponding to the degeneration where one of the boundary components of the domain shrinks to a point. This curve has a removeable singularity which when filled in yields $U_{\bullet}$, as shown in the center-left of Figure \ref{Fig:TrefoilDiskGluings}.
\ee

Let $a$ denote the oriented arc 
\begin{equation*}
a = \overline{\region_{1}} \cap \overline{\region_{2}},\quad \partial a = \pi_{\C}(\chord_{2}) - \pi_{\C}(\chord_{3}).
\end{equation*}
The images of each $u_{T} = \pi_{\C} \circ U_{T}$ will appear as the image of $u_{\bullet}$ with a slit removed, with the slit properly contained in the arc $a$. The boundary of such a slit will have two boundary branch points -- the unique critical points of the $u_{T}$.

Now suppose that we identify $a \simeq [0, 1]$ and consider the $2$-parameter family of holomorphic maps
\begin{equation*}
u_{a_{0}, a_{1}}: \dot{\Sigma} \rightarrow \C,\quad a_{0} < a_{1} \in [0, 1]
\end{equation*}
where $\dot{\Sigma}$ is an annulus with two boundary punctures on the same boundary component. The maps $u_{a_{0}, a_{1}}$ have two boundary critical points contained in the component of the annulus without boundary punctures, with critical values $a_{i} \in [0, 1]$. The image of each $u_{a_{0}, a_{1}}$ is $\im(u_{\bullet})$ with the slit $(a_{0}, a_{1}) \subset a$ removed.

\begin{figure}[h]
	\vspace{3mm}
	\begin{overpic}[scale=.9]{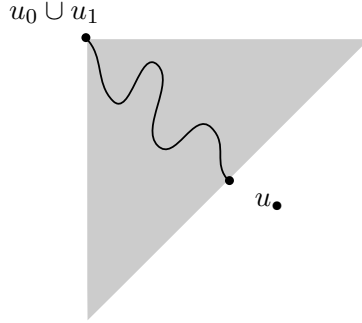}
		\put(-25, 105){$u_{0} \cup u_{1}$}
		\put(60, 40){$u_{\bullet}$}
	\end{overpic}
	\caption{The moduli space of $u_{a_{0}, a_{1}}$ is shown in gray, parameterized by the $a_{i}$ satisfying $0 < a_{0} < a_{1} < 1$. The collection of $u_{a_{0}, a_{1}}$ which we expect to lift to holomorphic maps $U_{T}$ is shown as the wiggly arc $b$. The boundary of the compactification of this arc contains the maps $u_{\bullet}$ (where $a_{0} = a_{1} \in \Int(a)$) and the union $u_{1} \cup u_{2}$ (where $0 =a_{0}, 1 = a_{1}$) in the boundary of the $u_{a_{0}}, a_{1}$ moduli space.}
	\label{Fig:UModuliSpace}
\end{figure}

The moduli space $\ModSpace_{a_{0}, a_{1}}$ of such $u_{a_{0}, a_{1}}$ is then an open triangle, naturally contained in the plane. The compactification of this space has $1$-dimensional strata provided by the subsets $\{ a_{0} =0 \}$, $\{ a_{1} = 1 \}$, and $\{ a_{0} = a_{1} \}$. There are at least two points in the boundary compactified moduli space $\overline{\ModSpace_{a_{0}, a_{1}}}$ which lift to holomorphic buildings in $\R^{4}$:
\be
\item The map $u_{0} \cup u_{1}$ lifts to the height $2$ holomorphic building $U_{1}
\cup U_{2}$ shown in the right most subfigure of Figure \ref{Fig:TrefoilDiskGluings}.
\item The map $u_{\bullet}$ lifts to the map $U_{\bullet}$ with an interior removable puncture deleted, shown in the center-left of Figure \ref{Fig:TrefoilDiskGluings}.
\ee

We expect -- assuming that $\obstruction$ is transverse at $0 \in \R$ -- that there should be a $1$-dimensional manifold $b \subset \ModSpace_{a_{0}, a_{1}}$ of maps which lift to the $U_{T}$ as shown. How can we see this using the obstruction class $\bigO$? Following the methodology of the previous example, we can integrate $\bigO$ over a loop $\gamma$ in $\Sigma$ which is parallel to the boundary component of $\Sigma$ without punctures.

Points in the boundary strata $\{ a_{0} \in (0, 1),\ a_{1} = 1\}$ of the compactified space $\overline{\ModSpace_{a_{0}, a_{1}}}$ may be viewed as nodal annuli obtained by identifying boundary punctures on the disk. These disks -- with the node removed -- lift to holomorphic maps in $\R^{4}$. Adding in the puncture then amounts to gluing a puncture which is negatively asymptotic to $\chord_{3}$ to a puncture which is positively asymptotic to $\chord_{3}$ in the domain. Because of our choice of orientation of $\gamma$ we see that $\int_{\gamma} dt\circ j \rightarrow \infty$ for points in $\ModSpace_{a_{0}, a_{1}}$ which are close to this boundary strata.

A similar argument shows that near the set $\{ a_{0} =  0,\ a_{1} \in (0, 1)\}$ in $\ModSpace_{a_{0}, a_{1}}$, the integral $\int_{\gamma} dt\circ j$ tends towards $-\infty$. Hence every every path in $\ModSpace_{a_{0}, a_{1}}$ of the form $a_{1} = a_{0} + \epsilon$ we see a count of $\pm 1$ holomorphic annuli in $\bigO^{-1}(0) = \ModSpace_{U}^{\Lambda}/\R$. Looking over all such paths as the parameter $\epsilon$ varies, we obtain the $1$-dimensional family shown by the black curve in Figure \ref{Fig:UModuliSpace}.

\subsection{Subtleties of compactified Dirichlet moduli spaces}\label{Sec:CompactifiedDirichlet}

Here we describe a moduli space of annuli which shows that the homeomorphism $\ModSpace^{\Dfancy} \rightarrow \ModSpace^{\lambda}$ of Proposition \ref{Prop:DirichletProjection} does not in general extend to a homeomorphism of the associated compactified moduli spaces. This behavior cannot be eliminated using our notions of simple diagrams.

\begin{figure}[h]
	\begin{overpic}[scale=.5]{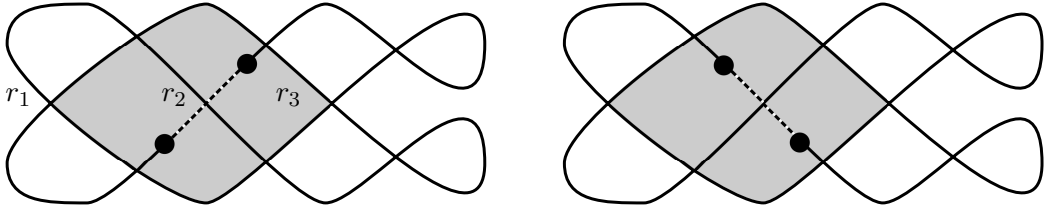}
		\put(0, 10){$\chord_{1}$}
		\put(15, 10){$\chord_{2}$}
		\put(26, 10){$\chord_{3}$}
	\end{overpic}
	\caption{Holomorphic annuli appearing with two boundary critical points.}
	\label{Fig:Compactification}
\end{figure}

Consider the holomorphic annuli in $\ModSpace^{\lambda}$ appearing in Figure \ref{Fig:Compactification}. The multicurve $\lambda$ depicted may be obtained from the Lagrangian resolution of a plat front for a left-right-simple Legendrian $\Lambda$ which is a stabilized unknot. The annuli have one boundary component with two boundary punctures which are positively asymptotic to double points associated to chords $\chord_{1}$ and $\chord_{3}$. The other boundary component of the annuli map to homotopically trivial curves in $\lambda$.

The associated moduli spaces $\ModSpace^{\lambda}$ and $\ModSpace^{\Dfancy}$ are homeomorphic to a disjoint union of two open triangles as described in the preceding subsection. The open codimension one strata of the associated compactified moduli spaces consists of
\be
\item nodal disks with a pair of identified nodes on the boundary of the disk. These occur when a single boundary critical point is pushed to the boundary of the image of a map in $\ModSpace^{\lambda}$.
\item half infinite cylinders with boundary punctures, or alternatively disks with two boundary punctures and a single interior puncture. These maps occur when the boundary critical points on the annulus coincide.
\ee
The codimension two boundary strata consists of limits of maps for which both boundary critical points tend towards the boundary of the image of a holomorphic map. We'll show that $\overline{\ModSpace^{\lambda}}$ is connected while $\overline{\ModSpace^{\Dfancy}}$ is not as depicted in Figure \ref{Fig:DisconnectedModuliSpace}, so that as claimed above there can be no homeomorphism $\overline{\ModSpace^{\Dfancy}} \rightarrow \overline{\ModSpace^{\lambda}}$.

\begin{figure}[h]
	\begin{overpic}[scale=.5]{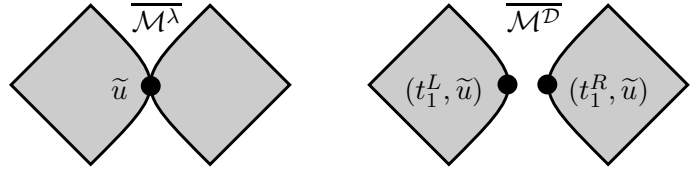}
		\put(18, 20){$\overline{\ModSpace^{\lambda}}$}
		\put(73, 20){$\overline{\ModSpace^{\Dfancy}}$}
		\put(15, 10){$\widetilde{u}$}
		\put(58, 10){$(t^{L}_{1}, \widetilde{u})$}
		\put(82, 10){$(t^{R}_{1}, \widetilde{u})$}
	\end{overpic}
	\caption{Compactified moduli spaces. The points associated to the map $\widetilde{u}$ discussed in the text are indicated by thickened dots.}
	\label{Fig:DisconnectedModuliSpace}
\end{figure}

Consider the punctured holomorphic disk $\widetilde{u} \in \overline{\ModSpace^{\lambda}}$ sending the interior marked point to the double point $\pi_{\C}(\chord_{2})$ of $\lambda$ associated to $\chord_{2}$. The map $\widetilde{u}$ can be seen as the $T \rightarrow 1$ limit of a family $\widetilde{u}^{L}_{T}, T \in (0, 1)$ of annuli shown in the left-hand side of Figure \ref{Fig:Compactification} with the boundary branch points both converging to $\pi_{\C}(\chord_{2})$. Likewise we can see $\widetilde{u}$ as a $T \rightarrow 1$ limit of annuli $\widetilde{u}^{R}_{T}, T \in (0, 1)$ as in the right-hand side of the figure, again with the boundary branch points converging to $\pi_{\C}(\chord_{2})$. Hence $\widetilde{u}$ lives in the boundary of the compactifications of each component of $\ModSpace^{\lambda}$, and so $\overline{\ModSpace^{\lambda}}$ is connected.

Now we show that $\overline{\ModSpace^{\Dfancy}}$ is disconnected. Let $t_{0}, t_{1}$ be the $t$-values endpoints of the chord $\chord_{2}$ of $\Lambda$ with $t_{0} < t_{1}$. If we lift the $\widetilde{u}^{L}_{T}$ to maps $(t^{L}_{T}, \widetilde{u}^{L}_{T}) \in \ModSpace^{\Dfancy}$, then in the $T \rightarrow 1$ limit we'll get a map $(t^{L}_{1}, \widetilde{u}^{L}_{1}) \in \overline{\ModSpace^{\Dfancy}}$ whose domain is a punctured disk. This map extends over the puncture, say $0 \in \disk$, so that $(t^{L}_{1}, \widetilde{u}^{L}_{1})(0) = (t_{0}, \pi_{\C}(\chord_{2}))$. This is a consequence of our conventions for crossings described in Section \ref{Sec:OvercrossingConvention}.

We can likewise lift the $\widetilde{u}^{L}_{R}$ to maps $(t^{L}_{R}, \widetilde{u}^{R}_{T}) \in \ModSpace^{\Dfancy}$ with a $T \rightarrow 1$ limit $(t^{R}_{1}, \widetilde{u}) \in \overline{\ModSpace^{\Dfancy}}$ whose domain is a punctured disk with interior puncture at $0 \in \disk$. Again the map extends over the puncture with $(t^{R}_{1}, \widetilde{u}^{R}_{1})(0) = (t_{1}, \pi_{\C}(\chord_{2}))$. Noting that both $(t^{L}_{1}, \widetilde{u})$ and $(t^{R}_{1}, \widetilde{u})$ project to the embedding $\widetilde{u} \in \overline{\ModSpace^{\lambda}}$ but take different values at the interior puncture, $0 \in \disk$, we conclude that $(t^{L}_{1}, \widetilde{u}) \neq (t^{R}_{1}, \widetilde{u})$. Hence $\overline{\ModSpace^{\Dfancy}}$ is disconnected.

We observe that at most one of $(t^{L}_{1}, \widetilde{u})$ and $(t^{R}_{1}, \widetilde{u})$ can lie in $\obstruction^{-1}(0)$. For if one such map exists, then the domain of the associated curve will be a disk with an interior puncture, determining a removable singularity. Extending the map $(s, t, \widetilde{u})$ over this puncture to obtain some $(s, t, u)$, the value of $t$ at $u^{-1}(\pi_{\C}(\chord_{2}))$ will be uniquely determined by the the values of $t$ along the boundary of the non-punctured disk. 

Generically -- within the space of Legendrian isotopies which preserve the isotopy class of $\lambda$ -- we can guarantee that neither of the $(t^{L}_{1}, \widetilde{u}), (t^{R}_{1}, \widetilde{u})$. To see this, we may keep the arcs of $\Lambda$ projecting to $\partial \im(\widetilde{u})$ fixed and perturb the $z$-value $\Lambda$ over the arcs projecting to the interior of $\im(\widetilde{u})$. Such perturbations will preserve the map $(s, t, u)$ and so the value of $t$ at $u^{-1}(\pi_{\C}(\chord_{2}))$.

Assuming transversality of $\obstruction$, an analysis similar to that of Section \ref{Sec:RigidUAnnulus} reveals that all points in $\partial \overline{\obstruction^{-1}(0)}$ will correspond to $1$-level SFT buildings whose domain is a disk with two boundary punctures positively asymptotic to $\chord_{1}$ and $\chord_{3}$ and a single, removable interior puncture. Likewise, a signed count of such curves can be seen to be $0$.

\subsection{Non-embedded annuli}

\begin{figure}[h]
	\begin{overpic}[scale=.3]{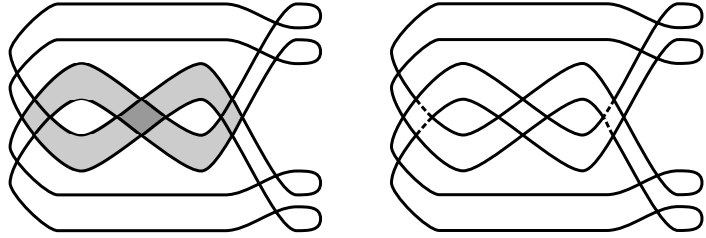}
	\end{overpic}
	\caption{On the left, a non-embedded annulus $u\in \ModSpace^{\lambda}$ on the Lagrangian resolution of the $2$-copy of the trefoil. On the right, a pair of dashed arcs parameterize the moduli space $\ModSpace^{\lambda}_{u}$.}
	\label{Fig:TrefoilTwoCopy}
\end{figure}

So far we have only considered examples of holomorphic annuli in $\ModSpace^{\lambda}$ which are embeddings when restricted to the interiors of their domains. Figure \ref{Fig:TrefoilTwoCopy} shows the Lagrangian resolution of a $2$-copy of the Legendrian trefoil of Section \ref{Sec:RigidUAnnulus}. On the left, we see a holomorphic annulus $u \in \ModSpace^{\lambda}$ with two boundary punctures on each connected component of the boundary of its domain. Two of the corners are non-convex and we may compute $\ind(u) = 2$. Clearly the annulus is not embedded in $\C$.

The moduli space $\ModSpace^{\lambda}_{u}$ is homeomorphic to an open square whose axes are parameterized by the images of boundary critical points or non-convex corners. These images must be contained in dashed arcs appearing in the right-hand side of Figure \ref{Fig:TrefoilTwoCopy}.

\begin{figure}[h]
	\begin{overpic}[scale=.3]{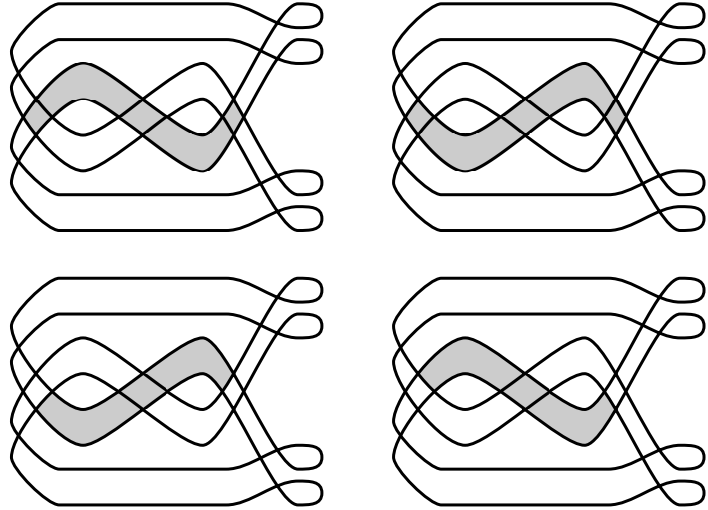}
	\end{overpic}
	\caption{Each column represents $2$-level SFT buildings of holomorphic disks appearing in $\partial \overline{\ModSpace^{\Lambda}}$ as Gromov limits of holomorphic annuli.}
	\label{Fig:TrefoilTwoCopyTwoLevel}
\end{figure}

Again, let's assume that $\obstruction$ is transverse to zero. An analysis analogous to that of Section \ref{Sec:RigidUAnnulus} reveals that $\partial \overline{\obstruction^{-1}(0)}$ consists of two points which are corners of the closed square $\overline{\ModSpace^{\lambda}_{u}}$. Each such point is associated to a $2$-level SFT building. The top level of each building is a holomorphic disk with two positive and two negative punctures. The bottom level of the building has two positive punctures which are glued to the negative punctures of the top level curve. Each column of Figure \ref{Fig:TrefoilTwoCopyTwoLevel} depicts such a $2$-level building.

\section{Restrictions on $J$-curves on left-right-simple $\Lambda$}\label{Sec:MainResults}

In this section we will prove Theorem \ref{Thm:Main}.

\begin{assump}
	Throughout this section we will take $\Lambda$ to be a left-right-simple Legendrian link in $\Rthree$ and use the notation $\lambda = \pi_{\C}(\Lambda) \subset \C$ throughout. We always assume that $u: \dot{\Sigma} \rightarrow \C$ is a holomorphic curve in $\ModSpace^{\lambda}$. The symbol $U$ will always denote a holomorphic curve $U = (s, t, u): \dot{\Sigma} \rightarrow \C$ in $\ModSpace^{\Lambda}$.
\end{assump}

We briefly outline the content of this section below while providing a broad overview of our proof:
\be
\item Section \ref{Sec:BoundaryPaths} sets up some notation and details basic restrictions on $u|_{\partial \dot{\Sigma}}$ imposed by the left-right-simple condition. 
\item Section \ref{Sec:AnotherIndex} uses the aforementioned restrictions to provide another means of computing Maslov indices which, unlike our previous computations, is specific to the left-right-simple setting. This will take care of Theorem \ref{Thm:Main} as it pertains to the case $\chi(\Sigma) < -1$.
\item Section \ref{Sec:DiskEmbeddings} establishes that all $\ind(u) = 0$ disks are embedded.
\item Section \ref{Sec:ChiOneElimination} bounds the indices of curves with $\chi(\Sigma) = 0, -1$.
\item Section \ref{Sec:AnnuliCriticalPoints} rules out the possibility of $\ind(u) = 2$ holomorphic annuli having interior critical points.
\item Section \ref{Sec:ProofCompletion} combines all of our results to complete our proof of Theorem \ref{Thm:Main}.
\ee

\subsection{Restrictions on boundary paths}\label{Sec:BoundaryPaths}

The foundation underlying the results of this section are restrictions which the left-right-simple condition impose on the rotation angles of paths in $\lambda$. We outline some terminology which will help speed up our arguments.

\begin{figure}[h]
	\begin{overpic}[scale=.5]{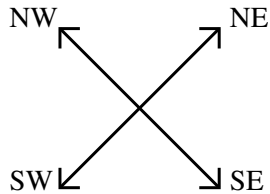}
		\put(-30, 0){SW}
		\put(105, 0){SE}
		\put(105, 100){NE}
		\put(-30, 100){NW}
	\end{overpic}
	\caption{Arc segments emanating from a self-intersection of $\lambda$.}
	\label{Fig:Directions}
\end{figure}

At each self-intersection of $\lambda$ there are four oriented arcs exiting the crossing at fixed angles, due to our requirement that $\Lambda$ is in good position. We call an arc-segment NE, NW, SW, or SE -- for northeast, northwest, southwest, or southeast -- according to Figure \ref{Fig:Directions}. Likewise, we say that a path in $\lambda$ which both begins and ends on self-intersections is, say, NE-SW if it exits its starting self-intersection from the NE and approaches its terminal self-intersection along the SW arc.

\begin{figure}[h]
\begin{overpic}[scale=.5]{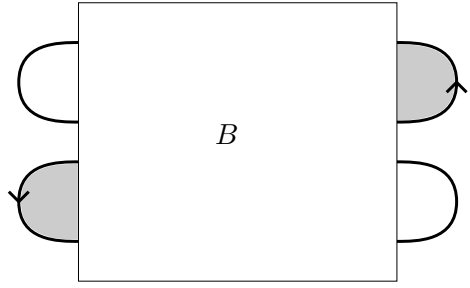}
\put(45, 30){$B$}
\end{overpic}
\caption{Ends of $\lambda$ determined by left and right pointing cusps, where $B$ is any positive braid with strange pointing left-to-right. The arrows indicate orientations on the boundaries of images of holomorphic maps.}
\label{Fig:Platification}
\end{figure}

We will be interested in paths in $\lambda$ which are the images of connected components of $\partial \dot{\Sigma}$ under holomorphic maps. We'd like to continue to think about the boundary behavior of $\ModSpace^{\lambda}$ curves as encoding the positivity and negativity of asymptotics of $\ModSpace^{\Lambda}$ curves.

\begin{prop}\label{Prop:Directions}
	If $u = \pi_{\C} \circ U$ for a holomorphic curve $U \in \ModSpace_{U}$ with $\Lambda$ left-right-simple, then for each connected component $\eta$ of $\partial \dot{\Sigma}$, inheriting its orientation from $\partial \Sigma$,
	\be
	\item if $\eta$ begins at a negative puncture, then $u(\eta)$ will exit the corresponding self-intersection of $\lambda$ in the NE or SW directions.
	\item if $\eta$ begins at a positive puncture, then $u(\eta)$ will exit the corresponding self-intersection of $\lambda$ in the NW or SE directions.
	\item if $\eta$ ends at a negative puncture, then $u(\eta)$ will enter the corresponding self-intersection of $\lambda$ in the NE or SW directions.
	\item if $\eta$ ends at a positive puncture, then $u(\eta)$ will enter the corresponding self-intersection of $\lambda$ in the NW or SE directions.
	\ee
\end{prop}

The above proposition is clear from looking at the under- and over-crossings of $\Lambda$ at a self-intersection of $\lambda$. When speaking of $\ModSpace^{\lambda}$ curves, we will continue to use the positive and negative puncture terminology as determined by the above proposition, even when a map $u$ is not determined by some $U$.

Next we address such paths which touch local maxima and minima of $x|_{\lambda}$.

\begin{prop}\label{Prop:BoundaryCusps}
	Let $u$ be a holomorphic map with domain $\dot{\Sigma}$ with $\lambda$ left-right-simple and let $\eta$ be a connected component of $\partial \dot{\Sigma}$ parameterized by a variable $T$, directed by the boundary orientation. If $u(\eta)$ touches a local minimum of $x|_{\lambda}$, then at this point $\frac{\partial y \circ u}{\partial T} < 0$ as shown on the left-hand side of Figure \ref{Fig:Platification}. If $u(\eta)$ touches a local maximum of $x|_{\lambda}$, then at this point $\frac{\partial y \circ u}{\partial T} > 0$ right-hand side of Figure \ref{Fig:Platification}.
\end{prop}

\begin{proof}
We work out the first statement regarding local minima of $x|_{\lambda}$. The proof of the second statement regarding local maxima is similar. Apart from looking at Figure \ref{Fig:Platification} to ensure that the signs of partial derivatives are correct, we only need to check that $\frac{\partial y \circ \eta}{\partial T}$ is non-zero when $u(\eta)$ touches a local minimum of $x|_{\lambda}$. If such a boundary critical point existed then the image of $u$ would spill over into the connected component of $\C \setminus \lambda$ of infinite area, which is impossible by our assumption that $\energy(u) < \infty$.
\end{proof}

\subsection{An index formula for left-right-simple diagrams}\label{Sec:AnotherIndex}

\begin{lemma}\label{Lemma:RotationAngleBounds}
Let $\eta= \eta(T)$ be a connected component of $\partial \dot{\Sigma}$ directed by the boundary orientation and let $u \in \ModSpace^{\lambda}$ be a holomorphic map with domain $\Sigma$ with $\lambda$ left-right-simple. If $\eta$ is not closed, then it may be parameterized by $T \in (0, 1)$ and we write $a_{i} \in \{ \pm \}, i=0, 1$ to indicate that the $T \rightarrow i$ punctures are positive or negative. With this notation, the rotation angle $\theta(u(\eta))$ is computed
\begin{equation*}
\theta(u(\eta)) = \frac{\pi}{2}\Big(\delta_{+, a_{1}} - \delta_{-, a_{0}} + 2\#(u(\eta)^{-1}(\Crit(x|_{\lambda}))) \Big)
\end{equation*}
where the $\delta$ are Kronecker deltas.
\end{lemma}

\begin{proof}
We use the boundary-relative homotopy invariance of the rotation angle to compute it from models.

\begin{figure}[h]
\begin{overpic}[scale=.35]{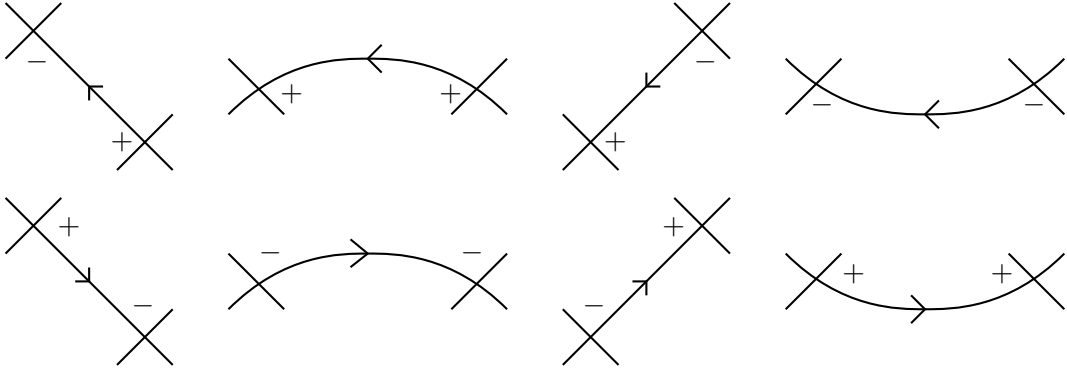}
\put(10, 20.5){$+$}
\put(2, 28){$-$}
\put(26, 25){$+$}
\put(41, 25){$+$}
\put(56.5, 20.5){$+$}
\put(65, 28){$-$}
\put(76, 24){$-$}
\put(96, 24){$-$}

\put(5, 12.5){$+$}
\put(12, 5){$-$}
\put(24, 10){$-$}
\put(43, 10){$-$}
\put(62, 12.5){$+$}
\put(54.5, 5){$-$}
\put(79, 8){$+$}
\put(93, 8){$+$}
\end{overpic}
\caption{Models for immersed paths which avoid local minima of $x|_{\lambda}$ and end at distinct self-intersections of $\lambda$. For left-to-right, the rotation angles in the first row are $0$, $\frac{\pi}{2}$, $0$, and $-\frac{\pi}{2}$. For the second row, the angles are $0$, $-\frac{\pi}{2}$, $0$, $\frac{\pi}{2}$.}
\label{Fig:CappingAwayFromMinMax}
\end{figure}

If $u(\eta)$ does not touch a local minima of $x|_{\lambda}$, then we can use one of the models from Figure \ref{Fig:CappingAwayFromMinMax} assuming the endpoints of the path correspond to distinct double points of $\lambda$. Each subfigure encodes the boundary angles and comparative $x$ values of immersed paths in $\lambda$ which avoid the local minima of $x|_{\lambda}$. Non-immersed paths can be ``pulled taut'' to eliminate any critical points of $x\circ u|_{\eta}$, so as to agree with one of these model paths. Any such deformation the path will leave not change the rotation angle or the sign of the associated puncture per Proposition \ref{Prop:Directions}. However pulling a path taut may change the direction that a path exists or enters a self-intersection of $\lambda$ from NE to SW, SW to NE, NW to SE, or SE to NW. Our formula for $\theta(u(\eta))$ may be verified by hand in each of these cases.

If the endpoints of the path correspond to the same double point and the path does not touch $\Crit(x|_{\lambda})$ then it must be homotopically trivial and have $0$ rotation angle. If $\eta$ is an open interval, then it must touch the double point at its ends with tangent vectors living in the same unoriented line at both of its ends. This Proposition \ref{Prop:BoundaryCusps} indicates that either $u(\eta)$  begins at a positive puncture and ends at a negative puncture or that it begins at a negative puncture and ends at a positive puncture. In either of these cases, the Kronecker delta terms in the statement of the current proposition cancel to yield $\theta(u(\eta)) = 0$ as we have shown to be true.

\begin{figure}[h]
\begin{overpic}[scale=.4]{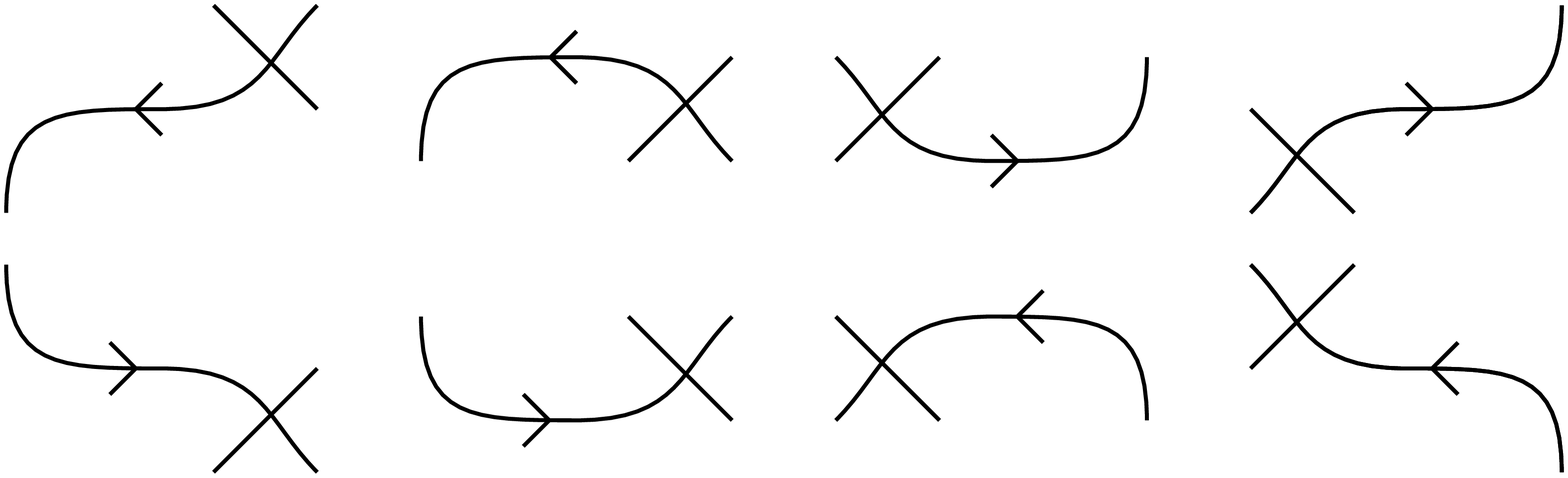}
\put(16, 22){$-$}
\put(16, 6){$-$}

\put(39, 6){$+$}
\put(39, 23){$+$}

\put(59, 23){$+$}
\put(59, 6){$+$}

\put(82, 23){$-$}
\put(82, 6){$-$}
\end{overpic}
\caption{Models for taught paths with one end on a double point of $\lambda$ and the other on $\Crit(x|_{\lambda})$. In the top row, paths begin at double points of $\lambda$ and end on points in $\Crit(x|_{\lambda})$. In the bottom row, paths begin at critical points of $x|_{\lambda}$ and end on double points of $\lambda$.}
\label{Fig:CappingTouchMin}
\end{figure}

In the case that $u|_{\eta}$ touches a local minimum of $x|_{\lambda}$ we look to Figure \ref{Fig:CappingTouchMin}. We cut the path into subpaths which begin or and at double points or $\Crit(x|_{\lambda})$. Each subpath is then pulled taught -- as described above -- so that $\frac{\partial u}{\partial T} \neq 0$ along the interior of the subpath. We then add up the contributions to rotation angle over each subpath. The figure encodes the relative $x$-values and endpoint type of each such subpath. The contribution of each subpath which begin at a local minimum of $x|_{\lambda}$ and and at a local maximum (or vice versa) is clearly $\pi$. The desired formula then follows.
\end{proof}

\begin{lemma}\label{Lemma:MaslovIndividualBoundary}
Suppose that $\partial \Sigma_{k}$ is a boundary component of $\Sigma$ and write $p^{\pm}_{i, k}$ for the boundary punctures of $\Sigma$ contained in $\partial \Sigma_{k}$ with superscripts indicating whether a puncture is positive or negative. Then for $\lambda$ left-right-simple, the Maslov index of $u|_{\partial \Sigma_{k}}$ may be computed
\begin{equation*}
\Maslov(u|_{\partial \Sigma_{k}}) + \#(p_{i, k}) = \#(u|_{\partial \Sigma_{k}}^{-1}(\Crit(x|_{\lambda}))) + \#(p_{i, k}^{+}).
\end{equation*}
\end{lemma}

\begin{proof}
We write $\eta_{i}$ for the connected components of $\partial \Sigma_{k}'$ and compute
\begin{equation}\label{Eq:MaslovToRotation}
\Maslov(u|_{\partial \Sigma_{k}}) + \#(p_{i, k}) = \frac{1}{\pi}\sum_{i}\Big( \theta(u|_{\eta_{i}}) + \frac{\pi}{2} \Big)
\end{equation}
using Equation \eqref{Eq:MaslovRotation}. If $\partial \dot{\Sigma}_{k}$ is a circle, then the proposition clearly follows from Lemma \ref{Lemma:RotationAngleBounds}.
	
If $\eta_{i}$ connect a negative punctures to a negative punctures, we may rephrase the first statement of Lemma \ref{Lemma:RotationAngleBounds} by writing
\begin{equation*}
\theta\left(u|_{\eta_{i}}\right) = \pi\#\left(u|_{\eta_{i}}^{-1}\left(\Crit\left(x|_{\lambda}\right)\right)\right) - \frac{\pi}{2}.
\end{equation*}
We apply analogous arguments to the remaining cases of the proposition.

In the case that $\eta_{i}$ connects a negative puncture to a positive puncture, then we can use the second statement in Lemma \ref{Lemma:RotationAngleBounds} to conclude that
\begin{equation*}
\theta\left(u|_{\eta_{i}}\right) + \frac{\pi}{2} = \frac{\pi}{2} + \pi\#\left(u|_{\eta_{i}}^{-1}\left(\Crit\left(x|_{\lambda}\right)\right)\right).
\end{equation*}
In the case that $\eta_{i}$ starts at a positive puncture and ends at a negative puncture, we can use the third statement in Lemma \ref{Lemma:RotationAngleBounds} to conclude that
\begin{equation*}
\theta\left(u|_{\eta_{i}}\right) + \frac{\pi}{2} = \frac{\pi}{2} + \pi\#\left(u|_{\eta_{i}}^{-1}\left(\Crit\left(x|_{\lambda}\right)\right)\right).
\end{equation*}
Finally, if $\eta_{i}$ connects two positive punctures, we can use the final case of Lemma \ref{Lemma:RotationAngleBounds} to conclude that
\begin{equation*}
\theta\left(u|_{\eta_{i}}\right) + \frac{\pi}{2} = \pi + \pi\#\left(u|_{\eta_{i}}^{-1}\left(\Crit\left(x|_{\lambda}\right)\right)\right).
\end{equation*}

We now plug the above equations into Equation \eqref{Eq:MaslovToRotation} to compute
\begin{equation*}
\Maslov\left(u|_{\partial \Sigma_{k}}\right) + \#\left(p_{i, k}\right) = \#\left(u|_{\partial \Sigma_{k}}^{-1}\left(\Crit\left(x|_{\lambda}\right)\right)\right) + \half \left( \#\left(- \rightarrow + \right) + \#\left(+ \rightarrow -\right) + 2\#\left(+ \rightarrow +\right)\right).
\end{equation*}
where the right-most summand counts the number of $\eta_{i}$ which end punctures with varying signs indicated by $+$ or $-$. This summand counts each positive puncture twice with multiplicity $\half$. Hence
\begin{equation*}
\half \left( \#\left(- \rightarrow + \right) + \#\left(+ \rightarrow -\right) + 2\#\left(+ \rightarrow +\right)\right) = \#\left(p_{i, k}^{+}\right),
\end{equation*}
completing the proof.
\end{proof}

We combine the results of the above proposition together with our standard index calculations and the fact that holomorphic maps have to have at least $1$ positive boundary puncture.

\begin{thm}\label{Thm:CritIndex}
When $\Lambda$ is left-right-simple, the index of a map $u$ with domain $\dot{\Sigma}$ may be computed
\begin{equation*}
\ind(u) = -2\chi(\Sigma) + \#\left(u|_{\partial \dot{\Sigma}}^{-1}\left(\Crit\left(x|_{\lambda}\right)\right)\right) + \#\left(p_{i}^{+}\right)
\end{equation*}
where the $p_{i}^{+}$ are the positive boundary punctures of $\Sigma$. Therefore the indices of maps $U$ satisfy
\begin{equation*}
\ind(U) = -\chi(\Sigma) + \#\left(u|_{\partial \dot{\Sigma}}^{-1}\left(\Crit\left(x|_{\lambda}\right)\right)\right) + \#\left(p_{i}^{+}\right)
\end{equation*}
\end{thm}

Some immediate consequences of the above theorem follow:
\be
\item If $\Sigma = \disk$, all $\ind(u)=0$ curves have $\leq 2$ positive punctures and all $\ind(u) = 1$ curves $u$ have $\leq 3$ positive punctures.
\item If $\Sigma$ is an annulus, then $\ind(u) \geq 1$ as there must be at least one positive boundary puncture.
\item If $\chi(\Sigma) < 0$, then 
\begin{equation*}
\ind(U) = \ind(u) + \chi(\Sigma) \geq -\chi(\Sigma) + 1.
\end{equation*}
In particular a $\chi(\Sigma) = -1$ curve has $\ind(U) \geq 2$ and a curve with $\chi(\Sigma) < -1$ has $\ind(U) \geq 3$.
\ee

Furthermore, it's easy to prove the following statement appearing in Theorem \ref{Thm:Main}:

\begin{cor}\label{Cor:FiniteDisks}
The reduced moduli space $\ModSpace^{\Lambda}_{1}/\R$ of $\ind(U) = 1$ holomorphic curves consists of a finite number of points for $\Lambda$ left-right-simple.
\end{cor}

\begin{proof}
	The collection of such disks is a discrete set up to biholomorphic reparameterization. There is a finite number of configurations of positive and negative asymptotics along the boundary of a disk. This follows from the facts that such a disk has at most two positive punctures, there are a finite number of double points of $\lambda$, and given a prescribed set of positive punctures there is a finite collection of possible choices of negative punctures due to the fact that a holomorphic disk must have positive area. 
	
	The collection of holomorphic disks with a given configuration of prescribed boundary asymptotics is also finite. Indeed, Gromov compactness \cite{SFTCompactness} tells us that (the closure of) this discreet space is compact, and hence finite.
\end{proof}

Holomorphic disks with $\ind(u) = 0$ will be further analyzed below in Section \ref{Sec:DiskEmbeddings}. There we will be provide an alternate, combinatorial proof of Corollary \ref{Cor:FiniteDisks}. In order to prove Theorem \ref{Thm:Main}, we will need to rule out the existence of $\ind(u) = 1$ annuli and $\ind(u) = 3$ curves with $\chi(\Sigma) = -1$. The required analysis is carried out in Section \ref{Sec:ChiOneElimination}.

\subsection{Embeddedness and finiteness of $\ind(u) =0$ disks}\label{Sec:DiskEmbeddings}

Here we will use Theorem \ref{Thm:CritIndex} together with the following lemma to establish that all $\ind(u) = 0$ holomorphic disks are embeddings into $\C$. We also reprove that there are only finitely many rigid disks without appealing to Gromov compactness \cite{SFTCompactness}.

\begin{lemma}\label{Lemma:DiskEmbedding}
	Suppose that $u: \disk \rightarrow \C$ is an immersion for which $x\circ u|_{\partial \disk}$ has exactly one local maxima and one local minima. Then $u$ is an embedding.
\end{lemma}

\begin{proof}
Consider the foliation of $\disk$ given by $\ker(u^{\ast}dx)$ whose leaves are the preimages of line segments of the form $\im(u) \cap \{ x = x_{0}\}$. The leaves corresponding to the maximum and minimum $x$-values are points. Using our assumption that $u$ is an immersion, every other leaf is a compact $1$-dimensional manifold along which $u^{\ast}dy \neq 0$.

By rescaling the $x$-coordinate of the target and reparameterizing $\disk$, we may assume without loss of generality that $x \circ u|_{\partial \disk}$ agrees with the $x$-coordinate on $\partial\disk \subset \C$ having unique boundary critical points $\pm 1 \in \disk \subset \C$. By our hypothesis on $x \circ u$, the restriction of $u$ to each connected component of $\partial\disk \setminus \{ \pm 1 \}$ is an embedding. If the images of these line segments intersect, say at some $x_{0} \in (-1, 1)$, with $y$-value $y_{0}$, then the restriction of $y\circ u$ to the leaf $\{ x = x_{0}\}$ must have (possibly degenerate) critical points as $y\circ u = y_{0}$ at the endpoints of the leaf. At such a critical point, $u$ cannot be an immersion. 

We conclude that $u|_{\partial \disk}$ is an embedding. Hence there is a single connected component of $\C \setminus u(\partial \disk)$ of finite area which must be the image of $\Int(\disk)$ under $u$. Hence $u$ is an embedding, completing the proof.
\end{proof}

\begin{figure}[h]
	\begin{overpic}[scale=.5]{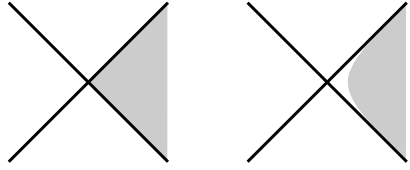}
	\end{overpic}
	\caption{Rounding convex corners. The rounding procedure can be applied by rotating the figure by any integer multiple of $\frac{\pi}{2}$.}
	\label{Fig:RoundCorner}
\end{figure}

\begin{prop}\label{Prop:HoloDiskEmbedding}
For $\lambda$ left-right-simple, each $\ind(u)=0$ holomorphic disk is an embedding.
\end{prop}

\begin{proof}
	By Theorem \ref{Thm:BranchIndex}, $u$ is an immersion and all of the boundary punctures of $u$ must correspond to convex corners. By trimming $\disk$ in neighborhoods about its boundary marked points as described in Figure \ref{Fig:RoundCorner}, we obtain a map $\widetilde{u}$ from a disk to $\C$ which is an immersion. It suffices to show that $\widetilde{u}$ is an embedding. The domain of $\widetilde{u}$ is strictly contained in $\disk$, but we will refer to it as $\disk$ for notational simplicity. The critical points of $x \circ \widetilde{u}|_{\partial \disk}$ correspond to the points in $\im(u)\cap \Crit(x|_{\lambda})$ together with the positive punctures of $u$.
	
	By Theorem \ref{Thm:CritIndex}, $u$ must be one of the following two forms:
	\be
	\item There are two positive punctures and $u$ is disjoint from $\Crit(x|_{\lambda})$.
	\item There is a single positive puncture and a single point of $\Crit(x|_{\lambda})$ touched by $u$.
	\ee
	In either case, we get exactly two critical points of $x \circ \widetilde{u}|_{\partial \disk}$. Hence by Lemma \ref{Lemma:DiskEmbedding}, $\widetilde{u}$ is an embedding.
\end{proof}

The embeddedness of $\ind(u) = 0$ disks facilitates an alternate proof of Corollary \ref{Cor:FiniteDisks} with an explicit (albeit, course) bound on the number of such disks.

\begin{proof}[Alternate proof of Corollary \ref{Cor:FiniteDisks}]
Suppose that $\C \setminus \lambda$ has $N \in \N$ connected components of finite area, indexed by $\{1, \dots, N\}$. The image of such a map $U$ will be uniquely determined by the components of $\C \setminus \lambda$ which are covered by $U$. For each connected component of $\C \setminus \lambda$ can be covered at most once and if two such covered connected components are adjacent, then our holomorphic map must extend over the arc lying at the intersection of their boundaries. These observations follow from Proposition \ref{Prop:HoloDiskEmbedding}.

Define a map
\begin{equation*}
\ModSpace^{\Lambda}_{1}/\R \rightarrow (\Z / 2\Z)^{N}
\end{equation*}
which assigns to $U$ a $1$ (or $0$) in the $j$th position of the target if the $j$th connected component of $\C \setminus \lambda$ is (or is not) covered by $\im(\pi_{\C}\circ U)$. The preceding arguments inform us that this map is injective. It then follows that $\#(\ModSpace^{\Lambda}_{1}/\R) \leq 2^{N}$, the cardinality of the target of this map.
\end{proof}

The following result will be useful later in the proof of Proposition \ref{Prop:NoLowIndexChiMinusOne}.

\begin{prop}\label{Prop:NoTouching}
Let $u \in \ModSpace^{\lambda}$ be an $\ind(u) = 0$ curve associated to a left-right-simple $\lambda$ whose domain $\dot{\Sigma}$ is a punctured disk with two positive punctures. Then the positive punctures can be labeled $p^{+}_{\min}, p^{+}_{\max}$ so that all negative punctures $p^{-}_{i}$ satisfy
\begin{equation*}
x(p^{+}_{\min}) < x(p^{-}_{i}) < x(p^{+}_{\max})
\end{equation*}
where the $x(p^{\ast}_{\ast})$ are the limiting $x$-values of the punctures under the map $u$.
\end{prop}

\begin{proof}
As all of the corners of $u$ must be convex, a positive puncture must point to the right or to the left. Let's say that that one of the corners corresponding to $p^{+}_{\max}$ points to the right. 

As we traverse $\partial\dot{\Sigma}$ leaving this puncture, we must exit the associated self-intersection in the NW direction with $x$ decreasing. Following the top row of Figure \ref{Fig:CappingAwayFromMinMax} we see that any negative punctures encountered on our way to $p^{+}_{\min}$ must have $x$-value strictly less than that of $p^{+}_{\max}$ and strictly greater than that of $p^{+}_{\min}$, which must be a left-pointing corner. 

The same goes for any negative punctures encountered on the way from $p^{+}_{\min}$ back to $p^{+}_{\max}$ while continuing to traverse the boundary of the punctured disk. The same line of argument applies if we had first assumed that our starting puncture corresponded to a left-pointing corner, up to a change in notation.
\end{proof}

\subsection{Elimination of $\chi = 0, -1$ curves with low index}\label{Sec:ChiOneElimination}

Here we analyze $\chi(\Sigma) = 0, -1$ curves.

\begin{prop}\label{Prop:NoAnnuli}
Suppose that $\Sigma$ is an annulus. Then any $u \in \ModSpace^{\lambda}$ with $\lambda$ left-right-simple has $\ind(u) \geq 2$. Therefore $U \in \ModSpace^{\Lambda}$ with $\Lambda$ left-right-simple has $\ind(U) \geq 2$.
\end{prop}

\begin{proof}
We will rule out the existence of $\ind(u) = 1$ annuli by contradiction. By Theorem \ref{Thm:CritIndex} such a curve could have only a single positive boundary puncture and not touch any point in $\Crit(x|_{\lambda})$. By Theorem \ref{Thm:BranchIndex}, such a curve can have either:
\be
\item a single non-convex corner for one of its punctures, all other punctures corresponding to convex corners, and no critical points or
\item all punctures corresponding to convex corners with a single boundary critical point and no interior critical points.
\ee

\begin{figure}[h]
\begin{overpic}[scale=.5]{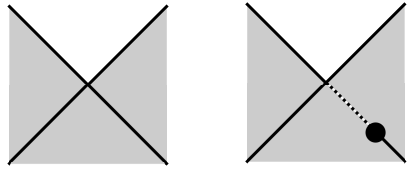}
\end{overpic}
\caption{Perturbing a non-convex corner to obtain a convex corner and a boundary critical point.}
\label{Fig:NonConvexSplitting}
\end{figure}

In the first case above, we may perturb $u$ near a non-convex corner to find a nearby holomorphic map which coincides with the second case described. This is described by Figure \ref{Fig:NonConvexSplitting} and its $\frac{\pi}{2}$ rotations. The right-most subfigure may also be reflected about a vertical line. After applying such a perturbation, it suffices to consider only the second case listed above.

\begin{figure}[h]
\begin{overpic}[scale=.5]{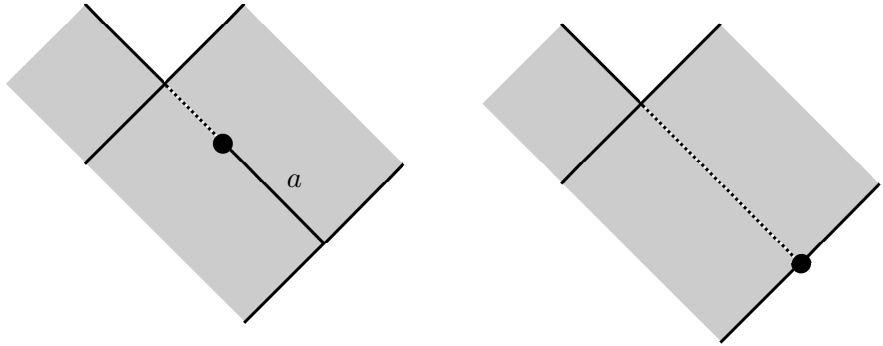}
\put(32, 18){$a$}
\end{overpic}
\caption{We push a boundary critical point to $\partial \Sigma$ along an arc $a$.}
\label{Fig:BoundarySplitting}
\end{figure}

Thus we may assume that we have a single boundary critical point away from which our map $u$ is an immersion. Following Figure \ref{Fig:BoundarySplitting}, we may push this boundary critical point along some arc $a \subset \Sigma$ for which
\be
\item $u(a) \subset \lambda$,
\item one endpoint of $a$ is the boundary branch point and the other lies on $\partial \dot{\Sigma}$.
\ee
In the limit we obtain a nodal curve whose domain can be described by compactifying $\dot{\Sigma} \setminus a$ and identifying two points in the boundary of the curve obtained. 

If $a$ separates $\Sigma$, then $\dot{\Sigma} \setminus a$ consists of an annulus and a disk component. This is impossible, as otherwise $u$ would determine a holomorphic map from the annulus to $\C$ which is immersed and hence have $\ind = 0$, a scenario ruled out by Theorem \ref{Thm:CritIndex}.

As $u$ is non-separating, we may compactify $\dot{\Sigma} \setminus a$ to obtain a disk with boundary punctures and a holomorphic map $\widetilde{u}$ from this disk to $\C$. The map $\widetilde{u}$ is immersed with all corners convex and cannot touch any critical points of $x|_{\lambda}$. Hence the disk has exactly $2$ boundary punctures, $p^{+}_{1}, p^{+}_{2}$ by the above theorem. As is clear from the cutting operation of Figure \ref{Fig:BoundarySplitting} -- specifically focusing on the dashed arc -- the map $\widetilde{u}$ cannot be an embedding near its boundary. This contradicts Proposition \ref{Prop:HoloDiskEmbedding}.
\end{proof}

Now we seek to show that all curves which are neither disks nor annuli have $\ind(U) = \ind(u) + \chi(\Sigma) \geq 3$ as per item (2) of Theorem \ref{Thm:Main}. According to Theorem \ref{Thm:CritIndex}, we need only consider the case $\chi(\Sigma) = -1$, so that $\Sigma$ is either a pair of pants or a genus $1$ curve with a single boundary component.

\begin{prop}\label{Prop:NoLowIndexChiMinusOne}
Suppose $\chi(\Sigma) = -1$. If $\lambda$ is left-right-simple, then $\ind(u) \geq 4$. Consequently if $\Lambda$ is left-right-simple, then $\ind(U) \geq 3$.
\end{prop}

\begin{proof}
We suppose that we have a map $u$ with $\ind(u) = 3$ and $\chi(\Sigma) = -1$, seeking to find a contradiction. Applying Theorem \ref{Thm:CritIndex}, we see that this is the minimum possible index and that our curve has a single positive puncture and cannot touch $\Crit(x|_{\lambda})$.

As in Figure \ref{Fig:BoundarySplitting}, we can perform a perturbation to make all punctures have branch order zero, possibly at the expense of introducing additional boundary critical points. If the given boundary puncture $p$ has $\ord_{u}(p) > 1$, then the perturbation will result in a single boundary critical point of order $1$, together with an additional boundary critical point of order $\ord_{u}(p)-1$. We will see how to deform boundary critical points of order two -- the worst possible case in our current setup -- into a pair of boundary critical points, each of order $1$ in the discussion around Figure \ref{Fig:BoundaryHalfDegen}, below.

Then we may assume that all punctures correspond to convex corners. Moreover, Theorem \ref{Thm:BranchIndex} tells us that out map has either
\be
\item a single interior critical point and a single boundary critical point of order $1$,
\item no interior critical points and boundary critical points whose orders sums up to $3$.
\ee

We first address the possibility that there exists an interior critical point. Suppose that an interior critical point exists and take a path in the moduli space which moves the critical point towards the boundary of $\partial \dot{\Sigma}$. In the limit there are two possible scenarios.

\begin{figure}[h]
\begin{overpic}[scale=.6]{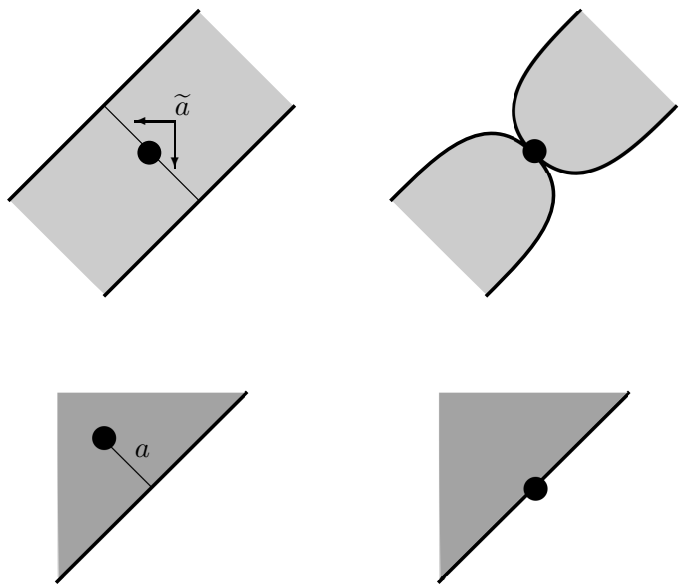}
\put(19, 19){$a$}
\put(25, 70){$\widetilde{a}$}
\put(25, 69){\vector(-1, 0){6}}
\put(25, 69){\vector(0, -1){7}}
\end{overpic}
\caption{Pushing an interior critical point to $\partial \Sigma$. Domains of curves appear in the top row with their images in the bottom row, a convention we will adopt for subsequent figures. The heavier shading indicates double-covering.}
\label{Fig:BoundaryNodalDegen}
\end{figure}

First, if the image of the critical point tends towards the boundary of the image of our holomorphic map, then a node will develop in the domain as shown in Figure \ref{Fig:BoundaryNodalDegen}. For arcs $a$ in this image which connect the image of the critical point to the boundary of the image will lift to double covers $\widetilde{a}$ as shown in the figure. In the limit, these arcs $\widetilde{a}$ may shrink to points where the nodes of interest will develop. Let's write $\widetilde{\Sigma}$ for the limiting curve obtained by removing the node and write $\widetilde{u}:\widetilde{\Sigma} \rightarrow \C$ for the associated holomorphic map. By our hypothesis on the map $u$, $\widetilde{u}$ is an immersion. If $\widetilde{\Sigma}$ has two connected components, then one of the components must have $\chi \leq 0$, and the index of the restriction of $\widetilde{u}$ to this surface is zero. This is impossible by Theorem \ref{Thm:BranchIndex}. If $\widetilde{\Sigma}$ is connected then it is an annulus and $\ind(\widetilde{u}) = 0$, which is again impossible.

\begin{figure}[h]
\begin{overpic}[scale=.5]{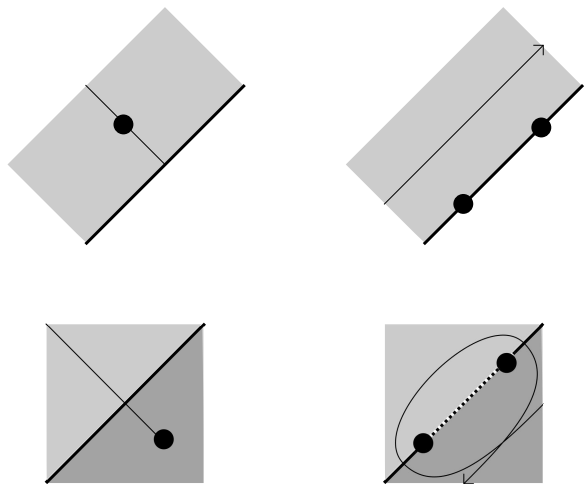}
\end{overpic}
\caption{In the left column we have a local picture of a map $u$ with a pair of $\ord_{u} = 1$ boundary critical points. To help intuit the branching behavior, we observe a thin arc in the domain being sent to a loop enclosing the images of the corresponding critical values via $u$. In the center column, the boundary critical points converge to a single $\ord_{u} = 2$ boundary critical point. In the right column, the single critical point moves from the boundary of the domain into its interior. Again we track the image of an arc under the map to help see the branching.}
\label{Fig:BoundaryHalfDegen}
\end{figure}

Second, the image of the critical point may stay within the interior of the image of the limiting map. In this case, the domain $\Sigma$ is topologically the same in the limit and an $\ord_{u} = 2$ critical point develops on the boundary of $\dot{\Sigma}$. The limit curve may also be realized as the limit of a sequence of curves having two $\ord_{u} = 1$ boundary critical points fuse as described in Figure \ref{Fig:BoundaryHalfDegen}.

We describe a simple model for this type of degeneration while explaining what is happening in the figure. Consider functions of the form
\begin{equation*}
u_{\epsilon} = \frac{z^{3}}{3} - \epsilon^{2}z,\quad du_{\epsilon} = (z + \epsilon)(z - \epsilon)
\end{equation*}
with domain the upper half plane $\Sigma = \{ y \geq 0 \}$. For $\epsilon \in \R \cup i\R$, the $u_{\epsilon}$ map $\partial \Sigma = \R$ to $\R$. For $\epsilon \in \R \setminus \{0\}$ there are two critical points in the boundary of the domain as shown in the left-hand side of Figure \ref{Fig:BoundaryHalfDegen}. For $\epsilon = 0$, these two boundary critical points converge as shown in the center column of the figure. Likewise taking the limit $\epsilon \rightarrow 0$ for $\epsilon \in i\R$ will push a single interior critical point towards $\partial \Sigma$. We see this as we transition from the right to center columns of the figure.

By pushing interior critical points to the boundary of $\dot{\Sigma}$, we can assume that $u$ has only convex corners, no interior critical points, and exactly three boundary critical points, each of order $1$. For the above argument can also be used to convert $\ord_{u} \geq 1$ boundary critical points into a collection of $\ord_{u} = 1$ boundary critical points.

\begin{figure}[h]
\begin{overpic}[scale=.5]{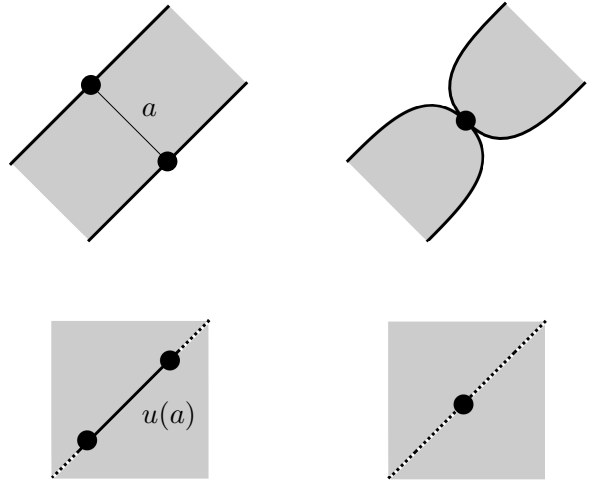}
\put(23, 63){$a$}
\put(23, 10){$u(a)$}
\end{overpic}
\caption{Nodal degeneration occurring upon the fusion of two boundary critical points. The arc $a$ connects boundary critical points in the domain of $u$.}
\label{Fig:BoundaryDoubleDegen}
\end{figure}

As in Figure \ref{Fig:BoundarySplitting}, we attempt to push one of the boundary critical points towards the end of a compact, connected arc $a \subset \dot{\Sigma}$ for which $u(a) \subset \lambda$. This time, because our curve has multiple boundary critical points, we may run into another boundary critical point as shown in Figure \ref{Fig:BoundaryDoubleDegen}. Then the limiting curve is a nodal domain which is obtained by shrinking the arc $a$ to a point. Let $\widetilde{\Sigma}$ be the associated curve with the node removed determining a holomorphic map $\widetilde{u}: \widetilde{\Sigma} \rightarrow \C$. By our hypothesis, the map $\widetilde{u}$ must have a single positive puncture and a single boundary critical point. Because there is only one positive puncture, $\widetilde{\Sigma}$ must be connected, and so $\widetilde{\Sigma}$ must be an annulus with a single boundary critical point, which is impossible by Proposition \ref{Prop:NoAnnuli}.

Therefore, our arc $a$ touches $\partial \dot{\Sigma}$ at both of its endpoints, with one endpoint on a boundary critical point and the other on a point of $\partial \dot{\Sigma}$ at which $u$ is an immersion. The arc $a$ is therefore described as in Figure \ref{Fig:BoundarySplitting}. If the arc $a$ separates $\dot{\Sigma}$, then either
\be
\item $\widetilde{\Sigma}$ is a disjoint union of a disk and an $\chi = -1$ surface. If this is so, the $\chi = -1$ surface will have at most two boundary critical points, no interior critical points, and so by the branch order index formula will have $\ind = 2$, which violates the estimate $\ind(u) \geq 3$ provided by Theorem \ref{Thm:CritIndex}.
\item $\widetilde{\Sigma}$ is a disjoint union of two annuli. In this case, either each annuli will have a single boundary critical point (which is impossible as the index of an annulus is at least $2$) or one annulus will have two boundary critical points and the other will have none (which is again impossible due to the branch order index formula).
\ee

Thus the arc $a$ cannot be separating so that $\dot{\Sigma} \setminus a$ is a connected annulus determining a holomorphic map with exactly two boundary critical points, two positive punctures, all punctures corresponding to convex corners, and no points touching $\Crit(x|_{\lambda})$. Moreover, there must be a pair $p_{1}^{+}, p_{1}^{-}$ consisting of a positive and negative puncture which are mapped to the same self-intersection of $\lambda$. Let's denote this map $u_{1}$ and rename the arc $a$ as $a_{1}$.

Let's try to cut the domain of $u_{1}$ with another arc $a_{2}$ for which $u_{1}(a_{2})\subset \lambda$ with an endpoint of $a_{2}$ touching a boundary critical point of $u_{1}$. We may assume that such an arc cannot fuse two boundary critical points as in Figure \ref{Fig:BoundaryDoubleDegen}, as we can consider such an arc as being contained in $\dot{\Sigma}$, a possibility ruled about by the above arguments. Thus the arc $a_{2}$ must extend to the boundary as in Figure \ref{Fig:BoundarySplitting}. If the arc $a_{2}$ is separating, then one of the connected components of $\dot{\Sigma} \setminus (a_{1} \cup a_{2})$ will be an annulus with less than two boundary critical points -- an impossibility by Proposition \ref{Prop:NoAnnuli}. Therefore $a_{2}$ must not separate.

Then $\Sigma_{2}' = \dot{\Sigma} \setminus (a_{1}\cup a_{2})$ is a disk and we have a map $u_{2}:\dot{\Sigma}_{2}\rightarrow \C$ for which
\be
\item all corners are convex,
\item the image of $u_{2}$ cannot touch $\Crit(x|_{\lambda})$,
\item there is a single boundary critical point of order $1$,
\item there are exactly three positive puncture $p^{+}_{1},p^{+}_{2},p^{+}_{3}$, and
\item there are negative punctures $p^{-}_{1},p^{-}_{2}$ for which $p^{+}_{i} = p^{-}_{i}$ for $i=1, 2$.
\ee
We can cut $\dot{\Sigma}_{2}$ using another arc $a_{3}$ which connects the remaining boundary critical point to $\partial \dot{\Sigma}_{2}$ and necessarily separates. We thus obtain $\Sigma_{3}' = \dot{\Sigma} \setminus (a_{1}\cup a_{2} \cup a_{3})$ -- a disjoint union of two disks -- and a holomorphic map $u_{3}:\dot{\Sigma}_{3}\rightarrow \C$ for which
\be
\item all corners are convex,
\item the image of $u_{3}$ cannot touch $\Crit(x|_{\lambda})$,
\item $u_{3}$ is an immersion,
\item there are exactly four positive puncture $p^{+}_{1},\dots, p^{+}_{4}$, and
\item there are negative punctures $p^{-}_{1},p^{-}_{2}, p^{-}_{3}$ for which $u_{3}(p^{+}_{i}) = u_{3}(p^{-}_{i})$ for $i=1, 2, 3$.
\ee
Since $u_{3}$ is an immersion not touching $\Crit(x|_{\lambda})$, each connected component of $\dot{\Sigma}_{3}$ has exactly two positive punctures by Theorem \ref{Thm:CritIndex}. By Proposition \ref{Prop:NoTouching}, for each pair $(p^{+}_{i}, p^{-}_{-})$ of punctures mapping to the same self-intersection of $\lambda$ we must have that $p^{+}_{i}$ and $p^{-}_{-}$ lie on distinct connected components of $\Sigma_{3}'$. Let's say that $p^{+}_{1}$ and $p^{+}_{2}$ live on the same connected component, possibly after a reindexing of the boundary punctures. Then again by Proposition \ref{Prop:NoTouching}, we must have that either the $x$-value of $p^{+}_{3}$ is less than all of the $x$-values of all negative punctures, or it must have $x$-value greater than that of all other punctures. 

This means that the pair $(p^{+}_{3}, p^{-}_{3})$ cannot possibly exist. Thus our $\ind(u) = 3$, $\chi(\Sigma) = -1$ curve cannot exist and the proof is complete.
\end{proof}

\subsection{Critical points of $\ind=2$ holomorphic annuli}\label{Sec:AnnuliCriticalPoints}

\begin{prop}\label{Prop:AnnuliInteriorCriticalPoints}
Suppose that $u \in \ModSpace^{\lambda}$ is an $\ind(u) = 2$ annulus with $\lambda$ left-right-simple. Then $u$ cannot have any interior critical points. Therefore an $\ind(U) = 2$ annulus cannot have interior critical points when $\Lambda$ is left-right-simple.
\end{prop}

\begin{proof}
Following the narrative of the preceding proofs, we suppose that such a map $u$ with an interior critical point exists and will arrive at a contradiction by studying nearby maps in $\ModSpace^{\lambda}_{u}$. By Theorem \ref{Thm:BranchIndex} there is only one such critical point, all punctures of $u$ correspond to convex corners, and the restriction of $u$ to $\partial \dot{\Sigma}$ is an immersion.

Following the proof of Proposition \ref{Prop:NoLowIndexChiMinusOne}, we attempt to push the interior critical point towards the boundary of $\im(u)$. If this is possible as described in Figure \ref{Fig:BoundaryNodalDegen}, then a node will develop in the domain corresponding to the pinching of an arc, $a \subset \Sigma$. Let's call the limiting map $\widetilde{u}$ whose domain is the limiting nodal curve, with nodal connections deleted. 

If $a$ is boundary parallel, then the nodal curve obtained will be the disjoint union of a disk and an annulus, and the restriction of $\widetilde{u}$ to the annulus will be an immersion with all punctures corresponding to convex corners, and hence have index $0$ by Theorem \ref{Thm:BranchIndex}. This is impossible by Proposition \ref{Prop:NoAnnuli}. If $a$ is not boundary parallel, the domain of $\widetilde{u}$ will be a connected disk and the map $\widetilde{u}$ will be an immersion having all boundary punctures corresponding to convex corners so that $\ind(\widetilde{u}) = 0$, again via application of Theorem \ref{Thm:BranchIndex}. However, following Figure \ref{Fig:BoundaryNodalDegen}, $\widetilde{u}$ cannot be an embedding, contradicting Lemma \ref{Lemma:DiskEmbedding}.

We have established that we cannot push the interior critical points to $\partial \im(u)$ as described in Figure \ref{Fig:BoundaryNodalDegen}, so we push the critical point to the end of a single sheet of $\im(u)$ as described in Figure \ref{Fig:BoundaryHalfDegen}. The modified curve $\widetilde{u} \in \ModSpace^{\lambda}_{u}$ has all punctures corresponding to convex corners and exactly two boundary critical points which are as depicted in the figure. Let $\eta$ be the connected component of $\partial \dot{\Sigma}$ in which the boundary critical points of $\widetilde{u}$ are contained.

Let $\widetilde{\lambda}$ be the connected component of $\lambda \setminus \Crit(x|_{\lambda})$ to which the curve $\eta$ is mapped via $\widetilde{u}$. We will attempt to push the boundary critical points of $\widetilde{u}$ along $\widetilde{\lambda}$. Without loss of generality, suppose that the critical point is as in Figure \ref{Fig:BoundaryHalfDegen}. Locally, the portion of $\im(\widetilde{u})$ which is multiply has greater $y$-coordinate values than the portion of $\im(\widetilde{u})$ which is singly covered. 

We push the right-most boundary critical point as far as we can to the right along $\widetilde{\lambda}$. In the limit we have a domain curve obtained from cutting an embedded arc, $b_{R}$, from $\widetilde{\Sigma}$ and a holomorphic map from this domain with only a single boundary critical point associated to the left-most critical point of Figure \ref{Fig:BoundaryHalfDegen}. If the arc $b_{R}$ is boundary parallel, then one of the components of the nodal domain obtained will be an annulus with $0$ or $1$ boundary critical points, in violation of Proposition \ref{Prop:NoAnnuli}. Hence the curve $b_{R}$ must be essential. The minimum $x$-value of $b$ is attained by the left-most boundary critical point of $\widetilde{u}$, so that $b_{R}$ does not touch the other critical point of $\widetilde{u}$. As all corners of $\widetilde{u}$ are convex, $b_{R}$ cannot end at a puncture. Note that the restriction of $\widetilde{u}$ to $b_{R}$ is an embedding.

\begin{figure}[h]
\begin{overpic}[scale=.5]{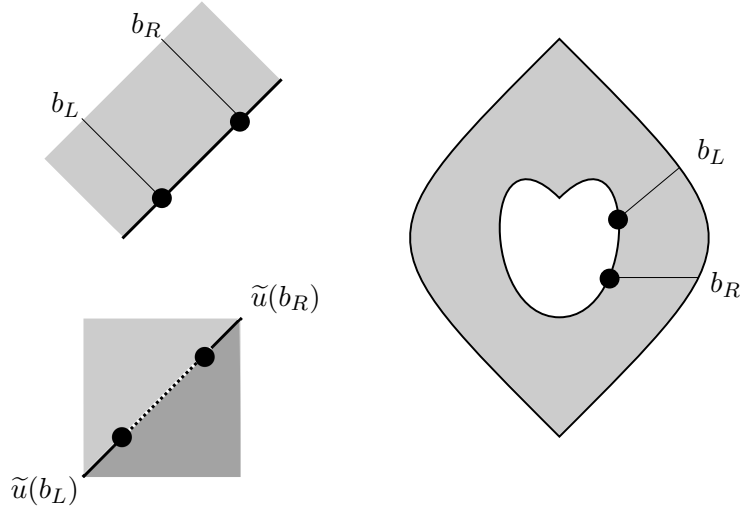}
\put(13, 67){$b_{R}$}
\put(1, 55){$b_{L}$}
\put(31, 26){$\widetilde{u}(b_{R})$}
\put(-5, -3){$\widetilde{u}(b_{L})$}
\put(100, 28){$b_{R}$}
\put(98, 48){$b_{L}$}
\end{overpic}
\vspace{3mm}
\caption{The arcs $b_{L}$ and $b_{R}$ (top-left), their images in $\C$ (bottom-left), and a schematic of their configuration within $\widetilde{\Sigma}$ (right). Compare with Figure \ref{Fig:BoundaryHalfDegen}.}
\label{Fig:BoundaryHalfDegenRedux}
\end{figure}

We may analogously define an arc $b_{L} \subset \widetilde{\Sigma}$ emanating from the left-most boundary critical point described above. The restriction of $\widetilde{u}$ to $b_{L}$ is also an embedding, and will be disjoint from $b_{R}$ as is seen by comparing their $x$-values under the map $\widetilde{u}$. Thus the complement of $b_{L} \cup b_{R}$ in $\widetilde{\Sigma}$ is a pair of disks. Our current situation is summarized in Figure \ref{Fig:BoundaryHalfDegenRedux}.

Compactifiying $\widetilde{\Sigma} \setminus (b_{L} \cup b_{R})$ appropriately, we obtain a pair of holomorphic maps from a pair of punctured disks to $\C$ with boundary on $\lambda$, with all boundary punctures corresponding to convex corners, and without critical points. Hence each holomorphic disk in $\ModSpace^{\lambda}$ has $\ind = 0$ and so must be embedded by Proposition \ref{Lemma:DiskEmbedding}. Looking at the lower-left portion of Figure \ref{Fig:BoundaryHalfDegenRedux}, we see that this is impossible. For if one of the disks covers the bottom triangle in the subfigure, then it must either do so twice, or cover both the upper and lower triangles. In either case our holomorphic disk is not an embedding.
\end{proof}

\subsection{Proof of Theorem \ref{Thm:Main}}\label{Sec:ProofCompletion}

We now combine the above results to complete our proof of Theorem \ref{Thm:Main}.

\textit{Every $\ind = 1$ holomorphic curve with boundary on $\R \times \Lambda$ is a disk with $1$ or $2$ positive punctures.} Theorem \ref{Thm:CritIndex} indicates that the domain of an $\ind(U) = 1$ curve is either a disk with $\leq 2$ positive punctures or and annulus. Then Proposition \ref{Prop:NoAnnuli} tells us that $\ind(U) = 1$ annuli cannot exist.

\textit{Every $\ind = 1$ holomorphic disk $U$ is such that $\pi_{\C}\circ U$ is an embedding. There are only finitely many such disks up to holomorphic reparameterization and translation in the $s$-coordinate.} This is the content of Section \ref{Sec:DiskEmbeddings}.

\textit{Every $\ind = 2$ holomorphic curve with boundary on $\R \times \Lambda$ is either a disk with at most $3$ positive punctures, or an annulus with at most $2$ positive punctures.} Theorem \ref{Thm:CritIndex} tells us that a $\Sigma = \disk$ curve of $\ind(U) = 2$ has at most $3$ positive punctures. It also rules out the possibility of $\ind(U) = 2$ curves with domain $\dot{\Sigma}$ having $\chi(\Sigma) < -1$. To rule out the existence of $\chi(\Sigma) = -1$ curves with $\ind(U) = 2$, apply Proposition \ref{Prop:NoLowIndexChiMinusOne}.

\textit{Every $\ind = 2$ holomorphic curve $U$ is such with $\pi_{\C} \circ U$ is simply covered and has no critical points in the interior of its domain.} By the above statement, we need only to consider disks and annuli. Following our existing notation, we write $u = \pi_{\C}\circ U$.

In the case of a disk, Theorem \ref{Thm:BranchIndex} tells us that if an interior critical point exists, then $\ind(u) \geq 2$ so that $\ind(U) \geq 3$. This establishes the second statement that $\pi_{\C}(U)$ has no interior critical points. The first statement that $\pi_{\C}(U)$ is simply covered can be derived from the first together with some consequences of the Riemann-Hurwitz formula: If a disk multiply covers a Riemann surface, then that surface must also be a disk and the covering must have at least one critical point, which we've already shown to be impossible.

In the case of an annulus, we have already established that $u$ has no interior critical points in Proposition \ref{Prop:AnnuliInteriorCriticalPoints}. If $u$ multiply covers some map $\widetilde{u}:\widetilde{\Sigma} \rightarrow \C$, then this results establishes that the covering must be free of critical points, so that $\widetilde{\Sigma}$ is also an annulus. We may derive, say from Theorem \ref{Thm:BranchIndex}, that
\begin{equation*}
\ind(u) = N\ind(\widetilde{u})
\end{equation*}
where $N \in \N$ is the covering multiplicity. By Proposition \ref{Prop:NoAnnuli}, if $N > 1$, then $\ind(U) = \ind(u) > 2$. This completes the proof of Theorem \ref{Thm:Main}.

\end{document}